\documentclass{amsart}

\usepackage[latin1]{inputenc}
\usepackage{tipa}
\usepackage{dsfont}

\usepackage{leftidx}

\usepackage{xcolor}

\usepackage{xr}
 
\usepackage{amsxtra}
\usepackage{amsmath}
\usepackage{amssymb}
\usepackage{amsfonts}
\usepackage[all,2cell,cmtip]{xy}
\UseAllTwocells
\usepackage{mathrsfs}
\usepackage{amsthm}
\usepackage{enumitem}
\usepackage[hidelinks]{hyperref}

\usepackage[capitalise]{cleveref}
\crefformat{enumi}{#2\textup{(#1)}#3}
\crefformat{equation}{(#2#1#3)}

\usepackage{subfigure}
\usepackage{todonotes}

\usepackage{euscript}

\setlist[enumerate]{leftmargin=*,labelindent=.5pc}
\usepackage{fullpage}

\newcommand{\nc}{\newcommand}

\nc{\renc}{\renewcommand}
\nc{\ssec}{\subsection}
\nc{\sssec}{\subsubsection}

\newtheorem{thm}[subsubsection]{Theorem}
\newtheorem{cor}[subsubsection]{Corollary}
\newtheorem{lem}[subsubsection]{Lemma}
\newtheorem{prop}[subsubsection]{Proposition}
\newtheorem{rem}[subsubsection]{Remark}

\numberwithin{equation}{section}

\newtheoremstyle{example}{\topsep}{\topsep}%
     {}
     {}
     {\bfseries}
     {.}
     {2pt}
     {\thmname{#1}\thmnumber{ #2}\thmnote{ #3}}

\theoremstyle{example}

\def\on{\operatorname}

\def\arr{\rightarrow}

\def\fib#1#2{{\rm fib}(#1 \arr #2)}

\newcommand{\sbt}{\,\begin{picture}(-1,1)(-1,-1)\circle*{2}\end{picture}\ }

\def\PP{ {\mathbb P}}

\def\QCoh{\on{QCoh}}
\def\Sh{\on{Sh}}

\def\C{{C}}

\def\D{ {D}}

\def\O{{\mathcal O}}

\setlist[enumerate,1]{label=(\arabic{*})}
\setlist[enumerate,2]{label=(\alph{*})}
\setlist[enumerate,3]{label=(\roman{*})}

\def\centerarc[#1](#2)(#3:#4:#5)  { \draw[#1] ($(#2)+({#5*cos(#3)},{#5*sin(#3)})$) arc (#3:#4:#5); }
\tikzstyle{dot}=[draw,circle,fill=black,inner sep=0,minimum size=4pt]
\tikzset{
    partial ellipse/.style args={#1:#2:#3}{
        insert path={+ (#1:#3) arc (#1:#2:#3)}
    }
}

\def\Mod#1{#1\on{-Mod}}
\def\Modd#1#2{#2\on{-Mod}}

\def\CircMod#1{#1\on{-Mod}^{S^{1}}}

\def\Loc#1{\on{Loc}(#1)}

\def\hoch#1{{\mathbf{HH}_{*}}(#1)}
\def\hochco#1{{\mathbf{HH}^{*}}(#1)}

\def\hochco#1{{\mathbf{HH}^{*}}(#1)}

\def\circinv#1{#1^{hS^{1}}}

\def\negcyc#1{{\mathbf{HC}^{-}_{*}}(#1)}

\def\Homm#1#2#3{\underline{\on{Hom}}_{#1}(#2,#3)}	

\def\Endd#1#2{\underline{\on{End}}_{#1}(#2)}

\def\End#1{{\on{End}}(#1)}

\def\IndCoh#1{\on{IndCoh}(#1)}
\def\QCoh#1{\on{QCoh}(#1)}

\def\Acal{\mathcal{A}}
\def\Mcal{\mathcal{M}}

\def\Abold{\mathbf{A}}

\def\Rbold{\mathbf{R}}

\def\Cbold{\mathbf{C}}
\def\Dbold{\mathbf{D}}

\def\LinCat{\on{LinCat}_{\Abold}}
\def\LinCatt#1{\on{LinCat}_{#1}}

\def\DGCat{\on{DGCat_{k}}}
\def\Stab{\on{Stab}}

\def\Arr#1{\on{Arr}(#1)}

\def\Vect{\on{Vect}_{k}}
\def\Id#1{\on{Id}_{#1}}

\def\ev#1{\on{ev}_{#1}}

\def\evL#1{\on{ev}^{l}_{#1}}
\def\evR#1{\on{ev}^{r}_{#1}}

\def\Inv#1{\on{Id}^{!}_{#1}}
\def\Serre#1{\on{Id}^{\vee}_{#1}}
    
\def\Hom#1#2#3{{\on{Hom}_{#1}}(#2,#3)}

\def\Mapp#1#2#3{{\on{Map}}_{#1}(#2,#3)}

\def\LMod#1{#1\on{-Mod}}

\def\Alg#1{\on{Alg}(#1)}

\def\AlgTwo#1{\on{Alg^{(2)}}(#1)}

\def\Spc{\on{Spc}}
\def\Spct{\on{Spct}}

\nc{\CV}{{\mathcal{V}}}
\nc{\CB}{{\mathcal{B}}}
\nc{\CM}{\mathcal{CM}}
\nc{\cD}{\mathcal{D}}
\nc{\CS}{\mathcal{S}}

\nc{\Mfld}{\CM\mathsf{fld}}
\nc{\Disk}{\cD{\mathsf{isk}}}
\nc{\Bsc}{\CB{\mathsf{sc}}}
\nc{\Snglr}{\CS{\mathsf{nglr}}}

\begin{document}    
    
\title{The cyclic Deligne conjecture and Calabi-Yau structures}
\author{Christopher Brav}
\address{C.B.: Centre of Pure Mathematics, Moscow Institute of Physics and Technology}
\author{Nick Rozenblyum}
\address{N.R.: Department of Mathematics, University of Chicago}

\maketitle
\begin{abstract} The Deligne conjecture (many times a theorem) endows Hochschild cochains of a linear category with the structure of an $E_{2}$-algebra, that is, of an algebra over the little $2$-disks operad. In this paper, we prove the cyclic Deligne conjecture, stating that for a linear category equipped with a Calabi-Yau structure (a kind of non-commutative orientation), the Hochschild cochains is endowed with the finer structure of a framed $E_{2}$-algebra, that is, of a circle-equivariant algebra over the little $2$-disks operad. Our approach applies simultaneously to both smooth and proper linear categories, as well as to linear functors equipped with a relative Calabi-Yau structure, and works for a very general notion of linear category, including any dualizable presentable $\infty$-category. As a particular application, given a compact oriented manifold with boundary $\partial M \subset M$, our construction gives chain-level genus zero string topology operations on the relative loop homology $H_{*}(LM,L\partial M)$. 
\end{abstract}


\tableofcontents

\section{Introduction}

The Hochschild cohomology $HH^{*}(R)$ of an associative algebra $R$, identified with the self-Ext groups $Ext^{*}_{R\otimes R^{op}}(R,R)$ of the diagonal bimodule, was shown by Gerstenhaber \cite{gerst} to admit further algebraic structure, including a graded commutative product
$$HH^{*}(R) \otimes HH^{*}(R) \stackrel{\sbt}\arr HH^{*}(R)$$ and a shifted Lie bracket
$$HH^{*}(R)[1] \otimes HH^{*}(R)[1] \stackrel{[\ ,\ ]}\arr HH^{*}(R)[1],$$
making $HH^{*}(R)$ a kind of shifted Poisson algebra, now known as a Gerstenhaber algebra. 

In a different direction, May \cite{may} established his ``recognition theorem"  identifying $n$-fold loop spaces with grouplike $E_{n}$-algebras in spaces, that is, spaces with an action of the little $n$-disks operad for inducing a group structure on connected components. Cohen \cite{cohen} then computed the homology of the little $n$-disks operads, which for little $2$-disks gives exactly the operad of Gerstenhaber algebras. In other words, the Hochschild cohomology $HH^{*}(R)$ of any associative algebra is naturally an algebra for the homology of the little $2$-discs operad. This observation led Deligne to conjecture in a 1993 letter that  Hochschild cochains should carry a natural structure of $E_{2}$-algebra.

This ``Deligne conjecture'' was then proved by various authors by a combinatorial approach. See for example Gerstenhaber-Voronov \cite{gerstvor}, Tamarkin \cite{tamarkinformality}, and McClure-Smith \cite{MS-Deligne}. In the combinatorial approach, one typically constructs various multilinear operations on a specific model for Hochschild cochains, checks that they satisfy relations coming from some chain operad, and then proves that this chain operad is equivalent to chains on the $E_{2}$-operad. While this approach is in some sense concrete, in that explicit formulas for the structure are provided relative to some specific model, it does not explain why Hochschild cochains should be an $E_{2}$-algebra or what purpose this structure serves. Moreover, it becomes increasingly difficult and sometimes impossible to extend the combinatorial approach from the case of strict associative algebras to more general situations arising in homotopy theory (for example, ring spectra) and algebraic geometry (for example, quasi-coherent sheaves on stacks). 

A more conceptual approach to the proof of the Deligne conjecture proceeds via an Eckmann-Hilton like argument, which provides Hochschild cochains with the structure of $2$-algebra (two homotopically coherent and compatible multiplications), then applies Dunn's additivity theorem \cite{dunn}  to identify $2$-algebras with $E_{2}$-algebras. The relevant arguments were worked out and refined in papers of Batanin \cite{batanin}, Tamarkin \cite{tamarkindg}, and Lurie \cite{luriefmp}, among others. 

The fundamental idea is to use the equivalence between $R$-bimodules and continuous endofunctors of the category of $R$-modules, with respect to which the diagonal bimodule $R$ gives the identity endofunctor. From this point of view, Hochschild cohomology is then cohomology of the derived endomorphism complex of the identity functor. More generally, for a linear $\infty$-category\footnote{such as the $\infty$-derived category of a ring $R$} $\Cbold$, Hochschild cochains $\hochco{\Cbold}$ is defined as the (derived) endomorphisms of the identity functor inside the linear $\infty$-category of all continuous endofunctors:
\[
\hochco{\Cbold} := \End{\Id{\Cbold}}.
\]
Since $\infty$-endofunctors form a monoidal $\infty$-category with composition as the monoidal product and with the identity functor as monoidal unit,  it is evident that Hochschild cochains $\hochco{\Cbold}$ should have two compatible multiplications, namely composite of endomorphisms and monoidal product of endomorphisms. More generally, one shows that the (derived) endomorphisms of the monoidal unit in a monoidal $\infty$-category always form a $2$-algebra, hence by Dunn additivity, an $E_{2}$-algebra.

As we explain in \Cref{nccalc}, this conceptual approach to the Deligne conjecture for $\hochco{\Cbold}$ has a number of advantages. In particular, it is manifestly Morita invariant, it generalizes to relative Hochschild cohomology $\hochco{\Cbold \arr \Dbold}$ for a functor, it characterizes Hochschild cochains $\hochco{\Cbold}$ as the universal $2$-algebra acting on $\Cbold$, and it leads to an invariant construction of the so-called non-commutative calculus for a dualizable category $\Cbold$, consisting of Hochschild cochains $\hochco{\Cbold}$ acting on Hochschild chains $\hoch{\Cbold}$ in a compatible manner with respect to the circle symmetry on the latter.

\medskip

In this paper, we establish the {\it cyclic Deligne conjecture}, a refinement of the Deligne conjecture for  $\infty$-categories equipped with a ``smooth Calabi-Yau structure''. Roughly, a smooth Calabi-Yau structure of dimension $d$ on a dualizable category $\Cbold$ is a circle-invariant class $\theta \in \hoch{\Cbold}^{hS^{1}}[-d]$ such that the action of Hochschild cochains (via the non-commutative calculus)
 gives  
a non-commutative Poincar\'e duality isomorphism
\[
\hochco{\Cbold} \stackrel{- \cap \theta}\simeq \hoch{\Cbold}[-d].
\]
By the Deligne conjecture, the lefthand side has an $E_{2}$-algebra structure, while the right hand side carries a natural circle action. The cyclic Deligne conjecture for smooth Calabi-Yau structures then states these two structures should be compatible in the sense that together they induce the structure of a circle-equivariant/framed $E_{2}$-algebra structure on $\hochco{\Cbold}$.  In particular, this induces a chain-level Lie algebra structure (and in fact, gravity algebra structure) on the \emph{cyclic chains} $\mathbf{HC}_{*}(\Cbold)$, refining and generalizing the Chas-Sullivan string Lie bracket in the case of string topology (see below).

In the body of the paper, we also consider variations of the cyclic Deligne conjecture for the dual notion of ``proper Calabi-Yau structure'', as well for smooth or proper relative Calabi-Yau structures on linear functors in the sense of Brav-Dyckerhoff \cite{bravdyck}.  A simplified version of our main theorem is the following.

\begin{thm}[Cyclic Deligne conjecture]
Let $\Cbold$ be a dualizable linear category (for example, a compactly generated DG or stable $\infty$-category) equipped with a Calabi-Yau structure (in either the smooth or proper sense). Then the Hochschild cochains $\hochco{\Cbold}$ carries an induced framed $E_{2}$-algebra structure refining the $E_{2}$-algebra structure provided by the solution of the Deligne conjecture. 

More generally, given a dualizable linear functor $\Cbold \stackrel{f}\arr \Dbold$ equipped with a relative Calabi-Yau structure (in either the smooth or proper sense), the relative Hochschild cohomology $\hochco{\Cbold \stackrel{f}\arr \Dbold}$ carries an induced framed $E_{2}$-algebra structure. 
\end{thm}

As an application of our main theorem, we establish the following result of interest in string topology.

\begin{thm}
Let $\partial M \subset M$ be a closed, oriented $d$-manifold with boundary. Then the relative loop homology $H_{*}(LM,L \partial M)[-d]$ carries a natural chain-level framed $E_{2}$-algebra structure
\end{thm}

In the case of empty boundary, this gives a chain-level refinement of the ``loop homology algebra'' of Chas-Sullivan \cite{chassull}, which moreover works for arbitrary coefficients. For further discussion of what was already known about chain-level string topology, see the discussion below of related work.

Other applications include the construction of a framed $E_{2}$-algebra structure for the relative Hochschild cochains for an anti-canonical divisor $Z \subset X$ in a variety (for example, for a cubic curve $E \subset \PP^{2}$) and for the ``non-commutative moment map'' into the path algebra of a doubled quiver, appearing in the theory of Nakajima quiver varieties.

Our approach to establishing the cyclic Deligne conjecture is conceptual rather than combinatorial. In \Cref{nccalc}, we review the Eckmann-Hilton type proof of the Deligne conjecture, following Lurie, and generalize it to the case of relative Hochschild cohomology $\hochco{\Cbold \stackrel{f}\arr \Dbold}$ of a linear functor. For dualizable functors $\Cbold \stackrel{f}\arr \Dbold$, we then construct the non-commutative calculus for relative Hochschild cochains $\hochco{\Cbold \arr \Dbold}$ acting on relative Hochschild chains $\hoch{\Cbold \arr \Dbold}$ and use this to formulate the notion of non-commutative relative orientation on a dualizable functor $\Cbold \stackrel{f}\arr \Dbold$ as a non-commutative analogue of a relative fundamental class from topology, with the role of relative homology replaced by relative Hochschild chains. In \Cref{factalg}, we use the theory of factorization algebras on stratified spaces, in the sense of Ayala-Francis-Tanaka \cite{AFT}, to decompose the category of framed $E_{2}$-algebras into simpler components. Namely, we establish the following.

\begin{thm}
The datum of a framed $E_{2}$-algebra is equivalent to the data of an $E_{2}$-algebra $R$, a circle-equivariant module $M$ over the Hochschild chains $\hoch{R}$, and a circle-invariant vector $m \in M^{hS^{1}}$ such that the composite $R \arr \hoch{R} \stackrel{- \cap m}\arr M$ is an isomorphism.
\end{thm}

Finally, in \Cref{cycdel}, we note that the ingredients in the above theorem necessary for building a framed $E_{2}$-algebra are provided precisely by the data of a relative Calabi-Yau structure. In this way, we obtain a general version of the cyclic Deligne conjecture, giving in particular the existence of framed $E_{2}$-algebra structures on Hochschild cochains arising in various examples from topology, algebraic geometry, and representation theory.

\medskip

\sssec{Related work}
To our knowledge, framed $E_{2}$-algebras related to Calabi-Yau structures first appeared in the guise of ``BV algebras'' in the context of topological conformal field theory and mirror symmetry. A BV algebra is  an algebra over the homology of the framed $E_{2}$-operad, and dg BV or homotopy BV algebras (in characteristic zero) are models for chain-level framed $E_{2}$-algebras, and typically arise as genus zero operations in some chain-level $2$d topological conformal field theory.  Among many relevant works, we mention \cite{wittenantibracket}, \cite{penkschw}, \cite{lianzuck}, \cite{getzlerbv}, \cite{barannkont}, and \cite{maninfrob}. Particularly relevant to us is Costello's construction \cite{costellotcft} of the all genus topological conformal field theory associated to a cyclic $A_{\infty}$-category over a field (a rigid version of what we call a proper Calabi-Yau category), as well as the alternative construction of the same structure by Kontsevich-Soibelman \cite{kontsoib}. Restricting this field theory to genus zero gives a framed $E_{2}$-algebra structure on Hochschild cochains of a cyclic $A_{\infty}$-category over a field. Alternatively, there has also been interesting work of Kauffmann \cite{kauf}  and Tradler-Zeinalian \cite{tradzein} for Frobenius algebras and Ward \cite{ward}  for cyclic $A_{\infty}$-categories, which give a more direct construction of the framed $E_{2}$-algebra structure on Hochschild cochains.


All of this work is in the context of proper Calabi-Yau categories, but arguably it is smooth Calabi-Yau categories that are more fundamental, with proper Calabi-Yau categories tending to arise as subcategories of smooth Calabi-Yau categories. A relatively simple yet rich example of a smooth Calabi-Yau category is that of (derived) local systems $\Loc{M}$ on a closed oriented manifold $M$. By the theorem of Goodwillie and of Jones \cite{goodwillie},\cite{jones}, the Hochschild chains of $\Loc{M}$ are circle-equivariantly equivalent to the chains $C_{*}(LM)$ on the free loop space $LM$, and an appropriate shift of the latter should therefore carry the structure of a framed $E_{2}$-algebra. At the level of homology, such a BV algebra structure was constructed by Chas-Sullivan \cite{chassull}, and chain-level models were constructed by Irie \cite{irie} and Drummond-Cole-Poirier-Rounds \cite{drummondstringtop}. Recently, Kontsevich-Takeda-Vlassoupolos \cite{ktv1,ktv2} have constructed chain-level all genus operations on the Hochschild chains of ``pre-Calabi-Yau categories'', which include the case of smooth Calabi-Yau categories over a characteristic zero field. Up to issues of unitality, the genus zero operations in this theory therefore endow shifted Hochschild chains of such a category with the structure of framed $E_{2}$-algebra. 

Our construction subsumes all of the genus zero examples above and enjoys a number of advantages. First, while the constructions just reviewed depend a priori on specific rigid, combinatorial models of the framed $E_{2}$ operad and the Hochschild cochains and typically require working over a ground field, our construction applies to very general $\infty$-categories in which rigid models need not apply. For instance, we are able to treat examples arising in stable homotopy theory or derived algebraic geometry. Second, our approach is manifestly invariant and at the same time computable, via the natural decomposition of a framed $E_{2}$-algebra into simpler components.  Finally, and perhaps most interestingly, our approach applies to relative Calabi-Yau structures, giving genuinely new examples of framed $E_{2}$-algebras arising in topology, algebraic geometry, and representation theory, including cases where the standard formalisms for constructing all genus field theories do not apply.

\ssec{Notation and Conventions}
\sssec{} In this paper, we work in the setting of $\infty$-categories.  We adopt the convention that ``category'' means $\infty$-category, and all categorical constructions are homotopy invariant, so for example ``limit'' means homotopy limit, all functors are ``total derived", and so on. 

Let $\Acal$ be a symmetric monoidal category and $\Mcal$ an $\Acal$-module category. Given two objects $x, y \in \Mcal$, consider the functor 
\begin{equation}\label{homobj}
\begin{gathered}
\Acal^{op} \arr \Spc \\
 a \mapsto \Mapp{\Mcal}{a \otimes x}{y}.
 \end{gathered}
 \end{equation}
When this functor is representable, the representing object is denoted $\Homm{\Acal}{x}{y} \in \Acal$. In the special case $x=y$, we write $\Endd{\Acal}{x}=\Homm{\Acal}{x}{x} \in \Acal$. By \cite[Sect. 4.7]{lurieHA}, $\Endd{\Acal}{x}$ is an associative algebra object in $\Acal$.

\sssec{}

Throughout this paper, we fix a presentably symmetric monoidal category $\Abold$, that is, a commutative algebra object in ${\rm Pr}^L$, the category of presentable categories and colimit preserving (aka continuous) functors endowed with the Lurie tensor product of presentable categories. See \cite[Chapter 1]{gr1} for a summary. 

In applications to topology, one often takes $\Abold=\Mod{E}$, the category of modules over a commutative ring spectrum $E$, and then $\LinCat$ is the category of presentable $E$-linear stable categories and $E$-linear functors. In particular, if $E$ is the sphere spectrum, then $\LinCat={\rm Stab}$, the category of presentable stable categories and continuous functors, while if $E=k$ a field, then $\LinCat=\DGCat$, the category of presentable $k$-linear DG categories and continuous $k$-linear functors. In some applications to algebraic geometry and representation theory, one takes $\Abold=\QCoh{X}$, the category of quasi-coherent sheaves on a prestack $X$. The reader can bear in mind these examples. 

\sssec{}

An object $\Cbold \in \LinCat$ is endowed with a continuous, coherently unital and associative action functor $\Abold \otimes \Cbold \arr \Cbold$, and a map between objects $f : \Cbold \arr \Dbold$ is a continuous functor coherently commuting with action functors. In particular, considering $\Abold$ as module over itself, the data of a continuous $\Abold$-functor $\Abold \arr \Cbold$ is given simply by any choice of image object $1_{\Abold} \in \Abold \mapsto x \in \Cbold$ for the symmetric monoidal unit $1_{\Abold} \in \Abold$. 

Note that by the adjoint functor theorem, any functor $\Abold \arr \Cbold$ has a right adjoint $\Homm{\Abold}{x}{-} : \Cbold \arr \Abold$. Evaluated on $x$, we obtain the endomorphism object $\Endd{\Abold}{x} \in \Abold$, which is endowed with a natural associative algebra structure in $\Abold$. 

\sssec{}

$\LinCat$ itself is symmetric monoidal with respect to the relative tensor product, denoted $\Cbold \otimes_{\Abold} \Dbold$, and possesses an internal Hom, denoted $\Homm{\LinCat}{\Cbold}{\Dbold}$.  

We denote by $\LinCat^{\rm dual} \subset \LinCat$ the subcategory of $\Abold$-linear categories that are dualizable over $\Abold$, that is, with respect to the relative tensor product, and $\Abold$-module functors $f: \Cbold \arr \Dbold$ with $\Abold$-linear right adjoints. We call such $\Abold$-module functors ``dualizable.''

\sssec{}

We emphasize that the dualizability condition on a linear category is $1$-categorical and takes place inside the category $\LinCat$ of big $\Abold$-linear categories. This condition should not be confused with $2$-dualizability/dualizability of small dg categories, a much more restrictive condition equivalent to a category simultaneously being smooth and proper. 

If $\Abold=\Vect$ is the category of DG vector spaces, for example, then $\LinCat=\DGCat$ is the category of presentable DG categories. Among the dualizable DG categories are all compactly generated DG categories, such as the familiar categories of DG modules $\LMod{R}$ over a DG algebra $R$ or quasi-coherent complexes $\QCoh{X}$ on a quasi-compact, quasi-separated scheme $X$, but also many non-compactly generated categories such  as the category of sheaves of DG vector spaces $\Sh(M)$ on a locally compact Hausdorff space $M$. 

For more on dualizability for linear categories, see \cite[Chapter 1]{gr1}.

\ssec{Acknowledgements}
The authors would like to thank Damien Calaque for drawing their attention to Horel's work on non-commutative calculus and factorization algebras and suggesting that framed $E_2$-algebras should be obtained in the same spirit.  We are also grateful to David Ayala, Dmitry Kaledin, Manuel Rivera, and Mahmoud Zeinalian for conversations about and interest in this project.
C.B. was partially supported by a ``Junior Leader'' grant no. 21-7-2-30-1 from the Basis Foundation.

\section{Noncommutative calculus for functors}\label{nccalc}

\ssec{$2$-algebras from endomorphisms of a monoidal unit/Eckmann-Hilton}

In this subsection, we review Lurie's approach to the Deligne conjecture for Hochschild cohomology (see the proof of \cite[Proposition 5.3.12]{luriefmp}), before generalizing the argument to relative Hochschild cohomology in the next subsection.

\sssec{}

As above, $\Abold$ denotes a fixed presentably symmetric monoidal category and $\LinCat$ denotes the category of $\Abold$-module categories with $\Abold$-linear functors.  We have the functor
$$ \on{Mod}: \Alg{\Abold} \to \LinCat $$
taking an associative algebra $R$ in $\Abold$ to the category $\Mod{R}$ of $R$-modules (in $\Abold$).

By \cite[Remark 4.8.5.17]{lurieHA}, this functor carries a natural symmetric monoidal structure. In particular, there is a natural equivalence $(\Mod{R}) \otimes_{\Abold} (\Mod{R^{'}}) \simeq \Mod{(R \otimes_{A} R^{'})}$. We therefore obtain an induced functor
\begin{equation}\label{2algtomoncat}
\AlgTwo{\Abold} := \Alg{\Alg{\Abold}} \arr \Alg{\LinCat}
\end{equation}
from $2$-algebras (``algebras in algebras") to monoidal $\Abold$-module categories. 

Roughly, given a $2$-algebra $R \in \AlgTwo{\Abold} \simeq \Alg{\Alg{\Abold}}$, the category $\Modd{\Abold}{R}$ of $R$-modules (in $\Abold$) is formed with respect to the ``inner" algebra structure of $R \in \AlgTwo{R}$, while the monoidal structure on $\Modd{\Abold}{R}$ uses the ``outer" algebra structure of $R \in \AlgTwo{\Abold}$ to form tensor products of $R$-modules.

\sssec{}
Now, let $\LinCat^{*} := (\LinCat)_{\Abold/}$ denote the category of \emph{pointed} $\Abold$-module
categories.  An object of $\LinCat^*$ is an $\Abold$-module category $\Cbold$ together with an object $x \in 
\Cbold$ and a morphism $(\Cbold, x) \to (\Cbold', x')$ is a morphism $f: \Cbold \to \Cbold'$ 
in $\LinCat$ together with an isomorphism $f(x) \simeq x'$.  By \cite[Theorem 2.2.2.4]{lurieHA}, $\LinCat^{*}$ is a symmetric monoidal category.  We have the symmetric monoidal functor
\[
\Alg{\Abold} \arr \LinCat^{*}
\]
taking an associative algebra $R$ in $\Abold$ to the pointed $\Abold$-module category $(\Mod{R},R)$, with right adjoint
\[
\LinCat^{*} \arr \Alg{\Abold}
\]
given by sending a pointed $\Abold$-module category $(\Cbold,x \in \Cbold)$ to the algebra of endomorphisms $\Endd{\Abold}{x} \in \Abold$.

\medskip
Now, we have the natural equivalence $\Alg{\LinCat} \simeq \Alg{\LinCat^*}$.  Therefore, the functor
$\AlgTwo{\Abold} \arr \Alg{\LinCat}$ is isomorphic to the functor
$$ \AlgTwo{\Abold} \simeq \Alg{\Alg{\Abold}} \to \Alg{\LinCat^*} \simeq \Alg{\LinCat} $$
and admits a right adjoint
\begin{equation}\label{endsofunit}
\Alg{\LinCat} \arr \AlgTwo{\Abold}
\end{equation}
given by sending a monoidal $\Abold$-module category $\Rbold$ to the endomorphism algebra $\Endd{\Abold}{1_{\Rbold}}$ endowed with a $2$-algebra structure. 

Roughly, $\Endd{\Abold}{1_{\Rbold}}$ is endowed with two compatible algebra structures, the ``inner" structure coming from composition of endomorphisms (this exists for endomorphisms of any object) and the ``outer" structure coming from tensor product of endomorphisms of the unit (which makes sense only for a unit object).

\sssec{}\label{sss:action of 2-alg}

Given an $\Abold$-linear monoidal category $\Rbold$ and an $\Abold$-linear category $\Cbold$, we may consider $\Rbold$-module structures on $\Cbold$, which by definition of endomorphism object identify with $\Abold$-linear monoidal functors $\Rbold \arr \Endd{\LinCat}{\Cbold}$. In particular, given a $2$-algebra $R$ in $\Abold$, we may consider actions of the monoidal category $\Modd{\Abold}{R}$ on $\Cbold$, in which case we speak simply of an action of the $2$-algebra $R$ on $\Cbold$. By the adjunction
$$
\xymatrix{
\AlgTwo{\Abold} \ar@<0.8ex>[r] & \ar@<0.8ex>[l] \Alg{\LinCat}
},$$
an $R$-action $\Modd{\Abold}{R} \arr \Endd{\LinCat}{\Cbold}$ is equivalent to a map of $2$-algebras $R \arr \hochco{\Cbold}$. In this sense, $\hochco{\Cbold}$ is characterized as the universal $2$-algebra acting on $\Cbold$.

\sssec{}
The above considerations together with Dunn's additivity theorem (see \cite[Theorem 5.1.2.2]{lurieHA}) establish the following version of the Deligne conjecture.

\begin{prop}\label{Deligne conjecture}
Let $\Cbold \in \LinCat$ be an $\Abold$-linear category, $\Endd{\LinCat}{\Cbold}$ the monoidal $\Abold$-module category of continuous $\Abold$-linear functors, whose monoidal structure is given by composite of endofunctors and whose monoidal unit $\Id{\Cbold}$ is the identity functor. The Hochschild cohomology object 
\[
\hochco{\Cbold} := \Endd{\Abold}{\Id{\Cbold}}
\]
carries a natural structure of $2$-algebra, in which the ``inner" algebra structure is composition of natural transformations and the ``outer" algebra structure is induced by composite of endofunctors. Moreover, $\hochco{\Cbold}$ is the universal $2$-algebra acting on $\Cbold$, in the sense that for any other $2$-algebra $R$, the datum of an action of $R$ on $\Cbold$ is equivalent to the datum of a map of $2$-algebras $R \arr \hochco{\Cbold}$.

By Dunn additivity, the $2$-algebra structure on $\hochco{\Cbold}$ is equivalent to an $E_{2}$-algebra structure.
\end{prop}

\sssec{}

We emphasize that, by definition, $\hochco{\Cbold}$ is a $2$-algebra in $\Abold$. In particular, if $\Abold=\Spct$, the category of spectra, then $\hochco{\Cbold}$ is what is sometimes called in the literature ``the topological Hochschild cohomology spectrum'', while if $\Abold=\DGCat$, then $\hochco{\Cbold}$ is the ``usual'' Hochschild cohomology complex over the ground field $k$.

\ssec{Hochschild cohomology of a functor}

In the previous subsection we have constructed Hochschild cohomology $\hochco{\Cbold}$ of an $\Abold$-linear category $\Cbold$ as the universal $2$-algebra in $\Abold$ acting on $\Cbold$. The goal of this subsection is to construct the Hochschild cohomology $\hochco{\Cbold \stackrel{f}\arr \Dbold}$ of an $\Abold$-linear functor $\Cbold \stackrel{f}\arr \Dbold$ as the universal $2$-algebra acting on source and target and making the arrow linear with respect to those actions.

\sssec{}

Consider the category of arrows $\Arr{\LinCat}$ of $\Abold$-module categories, whose objects are continuous $\Abold$-linear functors $f: \Cbold \arr \Dbold$ and whose morphisms are commutative squares
\begin{equation}\label{arrowsquare}
\xymatrix{\Cbold_{1} \ar[d]^{\Phi} \ar[r]^{f_{1}} & \Dbold_{1} \ar[d]^{\Psi} \\
\Cbold_{2} \ar[r]^{f_{2}} & \Dbold_{2},}
\end{equation}
that is, functors as above together with a natural isomorphism $f_{2} \circ \Phi \simeq \Psi \circ f_{1}$.

Note that $\Arr{\LinCat}$ has an evident pointwise symmetric monoidal structure with respect to which the source and target functors $s,t : \Arr{\LinCat} \arr \LinCat$ are symmetric monoidal. Moreover, the symmetric monoidal category $\LinCat$ acts on itself hence acts pointwise on the arrow category $\Arr{\LinCat}$, and the source and target functors are $\Abold$-linear.

\sssec{}

Given an object $\Cbold \stackrel{f}\arr \Dbold \in \Arr{\LinCat}$, we claim that the endomorphism object
\[
\Endd{\LinCat}{\Cbold \stackrel{f}\arr \Dbold} \in \LinCat
\]
exists. Indeed, in $\LinCat$ each of the internal Hom objects $\Homm{\LinCat}{\Cbold}{\Cbold}, \Homm{\LinCat}{\Dbold}{\Dbold}$, and  $\Homm{\LinCat}{\Cbold}{\Dbold}$ exists, and it is easy to check from the definition of the endomorphism object as representing the functor \Cref{homobj} that we have an equivalence of $\Abold$-module categories
\begin{equation}\label{endfibprod}
\Endd{\LinCat}{\Cbold \stackrel{f}\arr \Dbold} \simeq \Homm{\LinCat}{\Cbold}{\Cbold} \times_{\Homm{\LinCat}{\Cbold}{\Dbold}} \Homm{\LinCat}{\Dbold}{\Dbold}.
\end{equation}

By definition, the {\it relative Hochschild cohomology} of the functor $\Cbold \stackrel{f}\arr \Dbold$ is the $\Abold$-linear endomorphisms of the unit $\Id{f} \in \Endd{\LinCat}{\Cbold \stackrel{f}\arr \Dbold}$:
\[
\hochco{\Cbold \stackrel{f}\arr \Dbold}:=\Endd{\Abold}{\Id{f}}.
\]

\sssec{}

We have the following generalization of the Deligne conjecture.

\begin{prop}[Relative Deligne conjecture]\label{reldelconj}
Let $\Abold$ be a presentably symmetric monoidal category and $\Cbold \stackrel{f}\arr \Dbold$ an $\Abold$-linear functor. The relative Hochschild cohomology $\hochco{\Cbold \stackrel{f}\arr \Dbold}:= \Endd{\Abold}{\Id{f}}$
is the universal $2$-algebra acting on the arrow $\Cbold \stackrel{f}\arr \Dbold$, with actions on source and target given by natural maps of $2$-algebras $\hochco{\Cbold \stackrel{f}\arr \Dbold} \arr \hochco{\C}$ and $\hochco{\Cbold \stackrel{f}\arr \Dbold} \arr \hochco{\D}$. As a plain object of $\Abold$, there is a natural identification 
\[
\hochco{\Cbold \stackrel{f}\arr \Dbold} \simeq \hochco{\Cbold} \times_{\End{f}} \hochco{\Dbold}.
\]
\end{prop}

\begin{proof}
As an endomorphism object, $\Endd{\LinCat}{\Cbold \stackrel{f}\arr \Dbold} \in \LinCat$ carries a natural algebra structure, that is to say, an $\Abold$-linear monoidal structure, with monoidal unit $\Id{f}$ given by a commutative square \Cref{arrowsquare} having identity vertical arrows and identity natural transformation. 

As the endomorphisms of the unit $\Id{f}$ in the $\Abold$-linear monoidal category $\Endd{\LinCat}{\Cbold \stackrel{f}\arr \Dbold}$, $\hochco{\Cbold \stackrel{f}\arr \Dbold}$ carries a natural $2$-algebra structure in $\Abold$, via the functor \Cref{endsofunit}. Moreover, applying the symmetric monoidal source and target functors $s,t : \Arr{\LinCat} \arr \LinCat$, we obtain maps of $2$-algebras $\hochco{\Cbold \stackrel{f}\arr \Dbold} \arr \hochco{\Cbold}$ and  $\hochco{\Cbold \stackrel{f}\arr \Dbold} \arr \hochco{\Dbold}$, thus actions of $\hochco{\Cbold \stackrel{f}\arr \Dbold}$ on $\Cbold$ and $\Dbold$ for which the functor $\Cbold \stackrel{f}\arr \Dbold$ is linear. In short, $\hochco{\Cbold \stackrel{f}\arr \Dbold}$ ``acts on the arrow'' $\Cbold \stackrel{f}\arr \Dbold$.

In fact, by construction $\hochco{\Cbold \stackrel{f}\arr \Dbold}$ is the universal $2$-algebra in $\Abold$ acting on the arrow $\Cbold \stackrel{f}\arr \Dbold$. Indeed, given a $2$-algebra $R$ in $\Abold$, an action of $R$ on $\Cbold \stackrel{f}\arr \Dbold$ is by definition the datum of an $\Abold$-linear monoidal functor $\Modd{\Abold}{R} \arr \Endd{\LinCat}{\Cbold \stackrel{f}\arr \Dbold}$, which by \Cref{sss:action of 2-alg} is equivalent to the datum of a map of $2$-algebras $R \arr \hochco{\Cbold \stackrel{f}\arr \Dbold}$. 

Finally, note that via the equivalence \Cref{endfibprod}, we obtain an equivalence 
\begin{equation}\label{hochfibprod}
\hochco{\Cbold \stackrel{f}\arr \Dbold} \simeq \hochco{\Cbold} \times_{\End{f}} \hochco{\Dbold}.
\end{equation}

\end{proof}

\ssec{Non-commutative calculus for a functor}

The non-commutative calculus concerns the action of Hochschild cochains on Hochschild chains, and is the non-commutative analogue of the Cartan calculus of vector fields acting on differential forms. The relation between the two calculi is treated in \cite{bravroz2}.  The non-commutative calculus was developed by Tamarkin-Tsygan \cite{tamtsyg} and Kontsevich-Soibelman \cite{kontsoib} over a field and by Horel \cite{horel} for ring spectra and Iwanari \cite{iwanari} for compactly generated stable $\infty$-categories. 

In this subsection, we generalize the non-commutative calculus to the relative case of an $\Abold$-linear dualizable functor $\Cbold \stackrel{f}\arr \Dbold$, using functoriality properties of the trace.  


\sssec{}

Recall the subcategory $\LinCat^{\rm dual} \subset \LinCat$ 
of $\Abold$-module categories dualizable with respect to the tensor product over $\Abold$ and whose morphisms are ``dualizable'', that is, have $\Abold$-linear right adjoints. Given a dualizable $\Abold$-module category $\Cbold$, its {\it Hochschild homology} 
\[
\hoch{\Cbold}
\] 
is defined as the trace of the identity functor, that is, as the composite
\begin{gather*}
\Abold \stackrel{co}\arr \Cbold \otimes_{\Abold} \Cbold^{\vee} \simeq \Cbold^{\vee} \otimes_{\Abold} \Cbold \stackrel{ev}\arr \Abold
\end{gather*}
As an $\Abold$-linear endofunctor of  $\Abold$, the trace is completely determined by the image of $1_{\Abold}$ under the composite, and we shall simply identify $\hoch{\Cbold}$ with that object: $1_{\Abold} \mapsto \hoch{\Cbold}$. 

\sssec{}

When $\Abold=\Spct$, the category of spectra, then $\LinCat=\Stab$, the category of presentable stable categories. If $\Cbold \in \Stab$ is compactly generated, then in particular it is dualizable, and $\hoch{\Cbold}=THH(\Cbold)$, the topological Hochschild homology spectrum. For a more general dualizable stable category $\Cbold$, we take $\hoch{\Cbold}$ as the definition of the topological Hochschild homology spectrum. 

Likewise, when $\Abold=\Vect$, the category of DG vector spaces, then $\LinCat=\DGCat$, the category of presentable DG categories.  If $\Cbold \in \DGCat$ is compactly generated, then in particular it is dualizable, and $\hoch{\Cbold}$ is the ``usual'' Hochschild homology complex. For a more general dualizable DG category $\Cbold$, we take $\hoch{\Cbold}$ as the definition of ``usual''.

\sssec{}

By \cite[Theorem 2.14]{hoyetal}, Hochschild homology refines to a symmetric monoidal functor
\begin{equation}\label{hochhomfunct}
\hoch{-}: \LinCat^{\rm dual} \arr \Abold^{S^{1}} 
\end{equation}
from the category of dualizable $\Abold$-linear categories and dualizable $\Abold$-linear functors to the category of circle-equivariant objects in $\Abold$. 

In particular, given an $\Abold$-linear dualizable functor $\Cbold \stackrel{f}\arr \Dbold$, we obtain a circle-equivariant fiber sequence
\begin{equation}\label{relhochhomseq}
\hoch{\Cbold} \arr \hoch{\Dbold} \arr \hoch{\Cbold \stackrel{f}\arr \Dbold},
\end{equation}
where by definition the first arrow is given by the functor \Cref{hochhomfunct} and by definition the cofiber $\hoch{\Cbold \stackrel{f}\arr \Dbold}$ is the {\it relative Hochschild homology}.

\sssec{}

In order to precisely formulate the notion of ``non-commutative calculus'' of $\hochco{\Cbold \arr \Dbold}$ acting on $\hoch{\Cbold \arr \Dbold}$, we shall use the following.

\begin{lem}\label{hochhomof2alg} Let $R$ be a $2$-algebra in $\Abold$. Then $\hoch{R}$ is a circle-equivariant algebra in $\Abold$ and the Hochschild homology over $\Abold$ refines to a functor
\begin{equation}\label{abstcalc}
\hoch{-}: \LinCatt{R}^{\rm dual} \arr \CircMod{\hoch{R}}
\end{equation}
from $R$-linear dualizable categories 
to circle modules over $\hoch{R}$. Moreover, this refinement is natural for maps between $2$-algebras in $\Abold$.
\end{lem}

\begin{proof}
Via the functor \Cref{2algtomoncat}, the category $\Modd{\Abold}{R}$ of modules for a $2$-algebra $R$ in $\Abold$ has an induced monoidal structure. Note that $\Modd{\Abold}{R}$ is dualizable over $\Abold$ (with dual $\Modd{\Abold}{R^{op}}$), so $\hoch{R}$ is defined as a circle module in $\Abold$. Since the functor \Cref{hochhomfunct} is symmetric monoidal, the monoidal structure on $\Modd{\Abold}{R}$ induces a circle-equivariant algebra structure on $\hoch{R}$. Moreover, the construction is natural in maps of $2$-algebras via the functoriality of $R \mapsto \Modd{\Abold}{R}$ of \Cref{2algtomoncat} and the functoriality $\Modd{\Abold}{R} \mapsto \hoch{R}$ of \Cref{hochhomfunct}.
\end{proof}

\sssec{}

We define a {\it non-commutative calculus} over $\Abold$ to be a $2$-algebra $R$ together with a circle-module $M$ for the circle-equivariant algebra $\hoch{R}$. For the comparison of this structure to an alternative notion of non-commutative calculus in terms of the Kontsevich-Soibelman operad we refer to \cite{iwanari}. 
We have the following relative version of non-commutative calculus.


\begin{thm}[Relative non-commutative calculus]\label{relcalc}

Let $\Abold$ be a presentably symmetric monoidal category and $\Cbold \stackrel{f}\arr \Dbold$ a dualizable $\Abold$-linear functor. Then the circle-equivariant algebra $\hoch{\hochco{\Cbold \stackrel{f}\arr \Dbold}}$ in $\Abold$ acts on the circle-equivariant cofiber sequence
$$\hoch{\Cbold} \arr \hoch{\Dbold} \arr \hoch{\Cbold \stackrel{f}\arr \Dbold}.$$
In particular, there is a natural non-commutative calculus of relative Hochschild cochains acting on relative Hochschild chains.
\end{thm}

\begin{proof}

Consider now a dualizable functor between dualizable categories $\Cbold \stackrel{f}\arr \Dbold$ in $\LinCat$ 
By \Cref{reldelconj}, there is a universal $2$-algebra $\hochco{\Cbold \stackrel{f}\arr \Dbold}$ in $\Abold$ acting on the arrow $\Cbold \stackrel{f}\arr \Dbold$. In other words, we have a $2$-algebra $\hochco{\Cbold \stackrel{f}\arr \Dbold}$ and an arrow $\Cbold \stackrel{f}\arr \Dbold$ in $\LinCatt{\hochco{\Cbold \stackrel{f}\arr \Dbold}}^{\rm dual}$, so we may apply the functor \Cref{abstcalc} to obtain a circle-equivariant algebra $\hoch{\hochco{\Cbold \stackrel{f}\arr \Dbold}}$ in $\Abold$ and a map of circle-equivariant $\hoch{\hochco{\Cbold \stackrel{f}\arr \Dbold}}$-modules $\hoch{\Cbold} \arr \hoch{\Dbold}$, hence a cofiber sequence 
\[
\hoch{\Cbold} \arr \hoch{\Dbold} \arr \hoch{\Cbold \stackrel{f}\arr \Dbold}
\]
of circle-equivariant $\hoch{\hochco{\Cbold \stackrel{f}\arr \Dbold}}$-modules.

\end{proof}

\section{Framed $E_{2}$-algebras as equivariant factorization algebras}\label{factalg}

In this section, we use the theory of factorization algebras on stratified spaces of \cite{AFT} to give an algebraic description of framed $E_2$-algebras.

\ssec{Recollections on factorization algebras}

\sssec{}
Let $E_n$ denote the little $n$-disks operad, and let $\CV$ be a cocomplete symmetric monoidal category in which the tensor product commutes with colimits in each variable.  For 
any $n$-manifold $M$, let $\on{Alg}_M(\CV)$ denote the category of factorization algebras on $M$ 
valued in $\CV$, as defined in \cite[Sect. 5.4.5]{lurieHA}.  A framing on $M$ induces a functor
\begin{equation}\label{e:from diskalg to mfld}
\on{Alg}_{E_n}(\CV) \to \on{Alg}_M(\CV),
\end{equation}
where $\on{Alg}_{E_n}(\CV)$ is the category of $E_n$-algebras in $\CV$.

\medskip

In what follows, we will denote by $\on{Alg}(\CV) := \on{Alg}_{E_1}(\CV)$, the category of $E_1$-algebras (aka associative algebras).

\sssec{}
Now, the group $O(n)$ acts on the operad $E_n$.\footnote{In fact, the larger group $\on{Top}(n)$ of homeomorphisms of $\mathbb{R}^n$ acts (see
\cite[Remark 5.4.2.9]{lurieHA}) though unlike the action of $O(n)$ this is not obvious from the definition of $E_n$.}. The 
corresponding action on \Cref{e:from diskalg to mfld} is given by changes of framing of $M$.

In the case that $M = \mathbb{R}^n$,
\Cref{e:from diskalg to mfld} gives an equivalence (\cite[Example 5.4.5.3]{lurieHA})
\begin{equation}\label{e:disk algs as fact algs}
\on{Alg}_{E_n}(\CV) \simeq \on{Alg}_{\mathbb{R}^n}(\CV)
\end{equation}
which is equivariant with respect to the $O(n)$-action (where the action on the right hand side is induced
from the standard action of $O(n)$ on $\mathbb{R}^n$).\footnote{Note that for $\mathbb{R}^n$ changes of
framing are the same as isomorphisms of the ambient space.}

\sssec{Recollections on operads}
In what follows, we'll make use of the theory of (colored) operads in the setting of $\infty$-categories (as developed in \cite[Chapters 2 and 3]{lurieHA}).  We recall the salient features.

Given an operad $\mathcal{O}$, we will denote by $\mathcal{O}^{\otimes}$ the corresponding category over
$\on{Fin}_*$, the category of finite pointed sets.  We will also denote by $\mathcal{O}_{\langle 1 \rangle}$
the corresponding category of objects, i.e. the fiber of $\mathcal{O}^{\otimes}$ over $\langle 1 \rangle \in \on{Fin}_*$ (which is the set consisting of an element and a disjoint basepoint).

Suppose that $f: \mathcal{O} \to \mathcal{P}$ is a map of operads.  We then have the corresponding restriction map
$$ f^*: \on{Alg}_{\mathcal{P}}(\CV) \to \on{Alg}_{\mathcal{O}}(\CV). $$
We summarize its key aspects as follows.
\begin{prop}\label{p:properties of operadic restr}
The functor $f^*$ has the following properties:
\begin{enumerate}
\item\label{i:res pres sifted colims}
It preserves sifted colimits.
\item\label{i:conservativity of res}
If the functor $f_{\langle 1 \rangle}: \mathcal{O}_{\langle 1 \rangle} \to \mathcal{P}_{\langle 1 \rangle}$ is essentially surjective, then $f^*$
is conservative.
\item\label{i:operadic lke}
It admits a left adjoint
$$ f_!: \on{Alg}_{\mathcal{O}}(\CV) \to \on{Alg}_{\mathcal{P}}(\CV) .$$
called the \emph{operadic left Kan extension}.  For an object $x \in \mathcal{P}_{\langle 1 \rangle}$,
and $A \in \on{Alg}_{\mathcal{O}}(\CV)$, we have that
$$ f_!(A)(x) \simeq \on{colim}_{y \in (\mathcal{O}^{\otimes}_{\on{act}})_{/x}} A(y) $$
where $(\mathcal{O}^{\otimes}_{\on{act}})_{/x}$ is the slice category of the corresponding functor $\mathcal{O}^{\otimes}_{\on{act}} \to \mathcal{P}^{\otimes}_{\on{act}}$ on the subcategories consisting of active morphisms\footnote{Recall that a morphism in $\mathcal{O}^{\otimes}$ is active if its image in 
$\on{Fin}_*$ has the property that the only preimage of the basepoint is the basepoint.}, and the colimit is taken in $\CV$.
\item\label{i:operadic lke ff}
If the functor $\mathcal{O}^{\otimes} \to \mathcal{P}^{\otimes}$ is fully faithful,
then $f_!$ is fully faithful.
\end{enumerate}
\end{prop}
\begin{proof}
\ref{i:res pres sifted colims} follows from \cite[Proposition 3.2.3.1(4)]{lurieHA}.

\ref{i:conservativity of res} follows from \cite[Lemma 3.2.2.6]{lurieHA}.

\ref{i:operadic lke} is a special case of \cite[Cor. 3.1.3.5]{lurieHA}.

\ref{i:operadic lke ff} follows by \ref{i:operadic lke} and \ref{i:conservativity of res}.

\end{proof}

\sssec{}
For the sequel, we will need a generalization of the theory of factorization algebras to stratified
spaces. Following \cite{AFT}, let $\Snglr$ denote the symmetric monoidal\footnote{Recall that the tensor product is the disjoint union of underlying spaces, but it is \emph{not} the coproduct in $\Snglr$} $\infty$-category of (conically
smooth) stratified spaces and open embeddings, and let $\Bsc \subset \Snglr$ denote the full subcategory
of ``basics'' (see \cite[Sect. 1.1]{AFT}).  Given $X \in \Snglr$, let $\Snglr_X$ and $E_X$ denote operads given by
$$ \Snglr_X^{\otimes}:=(\Snglr_{/X})^{\sqcup} \underset{\Snglr^{\sqcup}}{\times} \Snglr^{\otimes} \mbox{ and } E_X^{\otimes}:=(\Bsc_{/X})^{\sqcup} \underset{\Snglr^{\sqcup}}{\times} \Snglr^{\otimes} $$
respectively, with the evident projection maps to $\on{Fin}_*$ (see \cite[Cor. 1.20]{AFT}).

Now, given $X \in \Snglr$, the category of factorization algebras on $X$ is defined as algebras over the $E_X$ operad:
$$ \on{Alg}_X(\CV) := \on{Alg}_{E_X}(\CV) .$$

\begin{rem}\label{r:def of fact alg}
The definition of $\on{Alg}_{X}(\CV)$ in \cite{AFT} is slightly different.  However, these notions are equivalent and it will be more convenient to work with the above definition.  Namely, in \cite{AFT}, they first consider the full subcategory $\Disk \subset \Snglr$ consisting of finite disjoint unions of basics, and corresponding operad $\Disk_X$ given by
$$ \Disk_X^{\otimes}:=(\Disk_{/X})^{\sqcup} \underset{\Snglr^{\sqcup}}{\times} \Snglr^{\otimes}. $$
They then define factorization algebras on $X$ as a certain full subcategory of $\on{Alg}_{\Disk_X}(\CV)$.  We have the evident map of operads given by $E_X^{\otimes} \to \Disk_X^{\otimes}$, which is fully faithful.  Therefore, by \Cref{p:properties of operadic restr}\ref{i:operadic lke ff}, the operadic left Kan extension functor
$$ \on{Alg}_{E_X}(\CV) \to \on{Alg}_{\Disk_{X}}(\CV) $$
is fully faithful.  Using \Cref{p:properties of operadic restr}\ref{i:operadic lke}, it is evident that the essential image is exactly the subcategory considered in \cite{AFT}.
\end{rem}

\sssec{}
The \emph{factorization homology} functor
$$ \int_{-} : \on{Alg}_{E_X}(\CV) \to \on{Alg}_{\Snglr_X}(\CV) $$
is defined as the operadic left Kan extension along the inclusion of operads $E_X \hookrightarrow \Snglr_X$.
This agrees with \cite[Definition 2.14]{AFT} by \Cref{r:def of fact alg}.

\medskip

Now, given a constructible bundle $f: X \to Y$, \cite[Lemma 2.24]{AFT} defines a map of operads
\begin{equation}\label{e:map of operads pushforward}
f^{-1}: E_Y \to \Snglr_{X}.
\end{equation}
This gives the pushforward, or \emph{relative factorization homology}, functor
$$ f_* : \on{Alg}_X(\CV) \to \on{Alg}_Y(\CV) $$
given as the composite
$$ \on{Alg}_X(\CV) \overset{\int_-}{\to} \on{Alg}_{\Snglr_X}(\CV) \overset{{(f^{-1})^*}}{\to} \on{Alg}_Y(\CV) .$$

\sssec{}
Suppose that $f: X \hookrightarrow Y$ is an open embedding; i.e. it is a morphism in $\Snglr$.  Composition with $f$ gives a map of operads $ E_X \to E_Y $, which we will also denote by $f$, slightly abusing notation.
Thus we have the corresponding restriction functor
$$ f^*: \on{Alg}_Y(\CV) \to \on{Alg}_X(\CV). $$
Now, we have a commutative diagram of operads
$$
\xymatrix{
E_X \ar[r]^f\ar[d] & E_Y \ar[d] \\
\Snglr_X \ar[r]^{\tilde{f}} & \Snglr_Y
}
$$
where the horizontal maps are given by composition with $f$.  Therefore, by adjunction, we have a natural transformation
\begin{equation}\label{e:pre base change}
\int_{-} \circ f^* \to \tilde{f}^* \circ \int_{-}: \on{Alg}_{Y}(\CV) \to \on{Alg}_{\Snglr_X}(\CV).
\end{equation}

\begin{prop}\label{p:pre base change iso}
The natural transformation \Cref{e:pre base change} is an isomorphism.
\end{prop}
\begin{proof}
By definition, we have that for any $Z\in \Snglr_{/X}$,
$$ ((E_X)^{\otimes}_{\on{act}})_{/Z} \simeq \Disk_{/Z} $$
and in particular doesn't depend on $X$.  Hence the functor
$$ ((E_X)^{\otimes}_{\on{act}})_{/Z} \to ((E_Y)^{\otimes}_{\on{act}})_{/Z} $$
is an equivalence, and the desired assertion follows by \Cref{p:properties of operadic restr}\ref{i:operadic lke}.
\end{proof}

\sssec{}
Now, suppose we have a commutative square of stratified spaces
$$ \xymatrix{
\tilde{X} \ar[r]^{g}\ar[d]_-{\overset{\circ}{p}} & \tilde{Y}\ar[d]^p\\
X \ar[r]^f & Y
}$$
where the vertical maps are constructible fibrations, the horizontal maps are open embeddings and $\tilde{X} = p^{-1}(X)$.  By construction of the map \Cref{e:map of operads pushforward} in \cite[Lemma 2.24]{AFT}, we have a commutative square of operads
$$
\xymatrix{
E_X \ar[r] \ar[d] & E_Y \ar[d] \\
\Snglr_X \ar[r] & \Snglr_Y
}
$$
Thus, by \Cref{p:pre base change iso} and the definition of relative factorization homology, we obtain the following base change formula:
\begin{prop}\label{p:base change}
In the situation above, there is a natural isomorphism
$$ \overset{\circ}{p}_* \circ g^* \simeq f^* \circ p_* : \on{Alg}_{\tilde{Y}}(\CV) \to \on{Alg}_{X}(\CV).$$
\end{prop}

\ssec{Factorization algebras on a cone}

\sssec{}
Let $X$ be a stratified manifold (in the sense of \cite{AFT}), and let $C(X)$ denote the cone on $X$.  We will give a convenient description of the category of factorization algebras on $C(X)$.

\medskip

By definition of $C(X)$, we have the commutative square
$$ \xymatrix{
X \times \mathbb{R}_{>0} \ar[r]^{j_X}\ar[d]_{\overset{\circ}{\pi}} & C(X) \ar[d]^{\pi} \\
\mathbb{R}_{>0} \ar[r]^-{j_{\on{pt}}} & \mathbb{R}_{\geq 0}=C(\on{pt})
},$$
where $\pi$ is the functor $C(-)$ applied to the map $X \to \on{pt}$; i.e., it is the projection to the cone 
coordinate.  Note that the horizontal maps are open embeddings and the vertical maps are constructible 
bundles.

By \Cref{p:base change}, there is a natural isomorphism of functors
\begin{equation}\label{e:base change}
 \overset{\circ}{\pi}_* \circ j_X^* \simeq j_{\on{pt}}^* \circ \pi_* : \on{Alg}_{C(X)}(\CV) \to \on{Alg}_{\mathbb{R}_{>0}}(\CV).
\end{equation}
Thus, we we obtain a functor
\begin{equation}\label{e:fiber product cone}
\on{Alg}_{C(X)} \to \on{Alg}_{X \times \mathbb{R}_{> 0}}(\CV) \underset{\on{Alg}_{\mathbb{R}_{>0}}(\CV)}{\times}  \on{Alg}_{\mathbb{R}_{\geq 0}}(\CV)
\end{equation}
given by $\pi_*$, $j_X^*$, and \Cref{e:base change}.
We now state the main theorem of this subsection:

\begin{thm}\label{t:algebras on cone}
The functor \Cref{e:fiber product cone} is an equivalence of categories.
\end{thm}

\sssec{}\label{sss:pointed modules}
In what follows, let $\on{AlgMod}_*(\CV)$ denote the category of pairs $(A,M)$, where $A \in \on{Alg}(\CV)$ is an associative algebra, and $M \in A\on{-mod}_{A/}$ is a \emph{pointed} module.

\medskip

As in \cite[Sect. 2.6]{AFT}, we have an equivalence
$$ \on{AlgMod}_*(\CV) \simeq \on{Alg}_{\mathbb{R}_{\geq 0}}(\CV). $$
We thus have:

\begin{cor}
The restriction functor
$$ \on{Alg}_{C(X)}(\CV) \to \on{Alg}_{X \times \mathbb{R}_{> 0}} (\CV)$$
is a Cartesian fibration with the fiber over $A \in \on{Alg}_{X \times \mathbb{R}_{\geq 0}}(\CV)$ equivalent to the category of pointed modules over $\int_{X \times \mathbb{R}_{>0}} A$.
\end{cor}

\begin{rem}
In fact, one can show that there is an equivalence of categories
$$ \on{Alg}_{X \times \mathbb{R}_{>0}}(\CV) \simeq \on{Alg}_{E_1}(\on{Alg}_X(\CV)).$$
Thus, this gives a purely algebraic description of $\on{Alg}_{C(X)}(\CV)$ in terms of
$\on{Alg}_{X}(\CV)$.
\end{rem}

\sssec{Proof of \Cref{t:algebras on cone}, Step 1}
The rest of this subsection is devoted to the proof of \Cref{t:algebras on cone}.  The strategy will be to relate both categories to a third one.

\medskip

Let $\on{Triv}$ be the trivial operad (see \cite[Examples 2.1.1.20 and 2.1.3.5]{lurieHA}), and let
\begin{equation}\label{e:cone point}
 \on{Triv} \to E_{C(X)}
\end{equation}
be the map which selects the object $C(X) \in \Bsc_{/X} \simeq (E_{C(X)})_{\langle 1 \rangle}$.
Taking the coproduct of \Cref{e:cone point} with the map $j_X: E_{X \times \mathbb{R}_{>0}} \to E_{C(X)}$, we obtain a map of operads
$$ \Phi: E_{X \times \mathbb{R}_{>0}} \boxplus \on{Triv} \to E_{C(X)} ,$$
where $E_{X \times \mathbb{R}_{>0}} \boxplus \on{Triv}$ is the coproduct operad.  By \cite[Theorem 2.2.3.6]{lurieHA}, it is given by
$$  (E_{X \times \mathbb{R}_{>0}} \boxplus \on{Triv})^{\otimes} \simeq  E_{X \times \mathbb{R}_{>0}}^{\otimes} \times \on{Triv}^{\otimes} \to \on{Fin}_* \times \on{Fin}_* \overset{\vee}{\to} \on{Fin}_* .$$

\medskip
Now, we have the restriction functor
$$ \Phi^*: \on{Alg}_{C(X)}(\CV) \to \on{Alg}_{E_{X\times \mathbb{R}_{>0}} \boxplus \on{Triv}}(\CV) \simeq \on{Alg}_{X\times \mathbb{R}_{>0}}(\CV) \times \CV .$$
By \Cref{p:properties of operadic restr}, it is conservative, preserves sifted colimits, and admits a left adjoint $\Phi_!$.  Therefore, it is monadic by Lurie's Barr-Beck theorem.  We will study the corresponding monad.

\begin{lem}\label{l:partial ff extension}
Let $p: \on{Alg}_{X \times \mathbb{R}_{>0}}(\CV) \times \CV \to \on{Alg}_{X \times \mathbb{R}_{>0}}(\CV)$ be the projection functor.
The composite
$$
p\circ (\on{id}_{\on{Alg}_{X \times \mathbb{R}_{>0}}(\CV) \times \CV} \overset{\eta}{\to} \Phi^* \circ \Phi_!)
$$
is an isomorphism of functors $\on{Alg}_{X \times \mathbb{R}_{>0}}(\CV) \times \CV \to \on{Alg}_{X \times \mathbb{R}_{>0}}(\CV)$, where $\eta$ is the unit of the adjunction.
\end{lem}
\begin{proof}
Let $\iota: E_{X \times \mathbb{R}_{>0}} \to E_{X \times \mathbb{R}_{>0}} \boxplus \on{Triv}$ denote the canonical map. Via the equivalence
$$  \on{Alg}_{E_{X\times \mathbb{R}_{>0}} \boxplus \on{Triv}}(\CV) \simeq \on{Alg}_{X\times \mathbb{R}_{>0}}(\CV) \times \CV, $$
the projection functor $p$ identifies with $\iota^*$.  Moreover, by definition of $\Phi$,
$$ \Phi \circ \iota \simeq j:= j_X.$$

\medskip

Now, for $d \in \Bsc_{/(X \times \mathbb{R}_{>0})}$, we have a commutative diagram
\begin{equation}\label{e:diagram of slices for ff}
\xymatrix{
((E_{X \times \mathbb{R}_{>0}})^{\otimes}_{\on{act}})_{/d} \ar[r] \ar[d] &
((E_{X \times \mathbb{R}_{>0}})^{\otimes}_{\on{act}})_{/j(d)} \ar[d] \\
((E_{X \times \mathbb{R}_{>0}} \boxplus \on{Triv})^{\otimes}_{\on{act}})_{/\iota(d)} \ar[r] &
((E_{X \times \mathbb{R}_{>0}} \boxplus \on{Triv})^{\otimes}_{\on{act}})_{/j(d)}
}
\end{equation}
Note that the functor $j: E_{X \times \mathbb{R}_{>0}}^{\otimes} \to E_{C(X)}^{\otimes}$ is fully faithful.  It follows that the top horizontal functor in \Cref{e:diagram of slices for ff} is an equivalence.

Now, for any object $a \in E_{C(X)}^{\otimes}$ such that $a\notin \on{Im}(j)$, we have that
$$ \on{Maps}_{E_{C(X)}^{\otimes}}(a, j(d)) = \emptyset .$$
It follows that the right vertical functor in \Cref{e:diagram of slices for ff} is an equivalence.  By the same logic, the left vertical functor in \Cref{e:diagram of slices for ff} is an equivalence as well.  Hence, the bottom horizontal functor
$$ ((E_{X \times \mathbb{R}_{>0}} \boxplus \on{Triv})^{\otimes}_{\on{act}})_{/\iota(d)} \to
((E_{X \times \mathbb{R}_{>0}} \boxplus \on{Triv})^{\otimes}_{\on{act}})_{/j(d)}$$
is also an equivalence.  Thus, by \Cref{p:properties of operadic restr}\ref{i:operadic lke}, we have that for any $A \in \on{Alg}_{X \times \mathbb{R}_{>0}}(\CV) \times \CV$, the natural map
$$ A(d) \to \Phi^* \Phi_!(A)(d)$$
is an isomorphism, as desired.
\end{proof}

\sssec{Step 2}
We will now show that the right hand side of \Cref{e:fiber product cone} is also monadic over $\on{Alg}_{X\times \mathbb{R}_{>0}}(\CV) \times \CV$.

Consider the functor
\begin{equation}\label{e:v factor on rhs}
\on{Alg}_{X \times \mathbb{R}_{> 0}}(\CV) \underset{\on{Alg}_{\mathbb{R}_{>0}}(\CV)}{\times}  \on{Alg}_{\mathbb{R}_{\geq 0}}(\CV) \to \on{Alg}_{\mathbb{R}_{\geq 0}}(\CV) \overset{\int_{R_{\geq 0}}}{\to} \CV,
\end{equation}
where the first functor is the projection.
Now, let
$$ \Psi: \on{Alg}_{X \times \mathbb{R}_{> 0}}(\CV) \underset{\on{Alg}_{\mathbb{R}_{>0}}(\CV)}{\times}  \on{Alg}_{\mathbb{R}_{\geq 0}}(\CV) \to \on{Alg}_{X\times \mathbb{R}_{>0}}(\CV) \times \CV $$
denote the product of the projection and \Cref{e:v factor on rhs}.  By \Cref{p:properties of operadic restr}, $\Psi$ is conservative and preserves sifted colimits.

Now, for $(A,v) \in \on{Alg}_{X\times \mathbb{R}_{>0}}(\CV) \times \CV$ by \Cref{l:partial ff extension} (applied to $X = \on{pt}$), we have an isomorphism
$$ j_{\on{pt}}^* ((\Phi_{\on{pt}})_!(\overset{\circ}{\pi}_*(A), v)) \simeq \overset{\circ}{\pi}_*(A) ,$$
where
$$ \Phi_{\on{pt}}: E_{\mathbb{R}_{>0}} \boxplus \on{Triv} \to E_{C(\on{pt})} \simeq E_{\mathbb{R}_{\geq 0}} $$
is the corresponding map of operads (given by $j_{\on{pt}}$ and $\mathbb{R}_{\geq 0} \in \Bsc_{/\mathbb{R}_{\geq 0}}$).  We obtain that $\Psi$ admits a left adjoint given by
\begin{equation}\label{e:left adjoint of psi}
\Psi^L: (A,v) \mapsto (A, (\Phi_{\on{pt}})_!(\overset{\circ}{\pi}_*(A), v)).
\end{equation}
Hence, $\Psi$ is monadic.

\sssec{Step 3}
By definition, restriction along \Cref{e:cone point} is the functor
$$ \int_{C(X)}: \on{Alg}_{C(X)}(\CV) \to \CV .$$
By \cite[Theorem 2.25]{AFT}, it factors as
$$ \on{Alg}_{C(X)}(\CV) \overset{\pi_*}{\to} \on{Alg}_{\mathbb{R}_{\geq 0}}(\CV) \overset{\int_{\mathbb{R}_{\geq 0}}}{\to} \CV. $$
Hence, we have a commutative diagram
$$ \xymatrix{
\on{Alg}_{C(X)}(\CV) \ar[rr]^-{\Cref{e:fiber product cone}} \ar[dr]_-{\Phi^*} & & \on{Alg}_{X \times \mathbb{R}_{> 0}}(\CV) \underset{\on{Alg}_{\mathbb{R}_{>0}}(\CV)}{\times} \on{Alg}_{\mathbb{R}_{\geq 0}}(\CV) \ar[dl]^-{\Psi} \\
& \on{Alg}_{X \times \mathbb{R}_{> 0}}(\CV) \times \CV
}$$
Since both $\Phi^*$ and $\Psi$ are monadic, we obtain a map of monads
\begin{equation}\label{e:map of monads}
\Psi \circ \Psi^L \to \Phi^* \circ \Phi_!.
\end{equation}
To show that \Cref{e:fiber product cone} is an equivalence, it suffices to show that \Cref{e:map of monads}
is an isomorphism of underlying endo-functors.  By \Cref{l:partial ff extension}, we have that
$$ p \circ (\on{id}_{\on{Alg}_{X \times \mathbb{R}_{>0}}(\CV) \times \CV} \overset{\eta}{\to} \Phi^* \circ \Phi_!) $$
is an isomorphism, and by \Cref{e:left adjoint of psi}, we have that
$$ p \circ (\on{id}_{\on{Alg}_{X \times \mathbb{R}_{>0}}(\CV) \times \CV} \overset{\eta}{\to} \Psi \circ \Psi^L) $$
is also an isomorphism.  It follows that $p\circ \Cref{e:map of monads}$ is an isomorphism.

\sssec{Step 4: Conclusion}
It remains to show $q \circ \Cref{e:map of monads}$ is an isomorphism, where
$$ q: \on{Alg}_{X \times \mathbb{R}_{> 0}}(\CV) \times \CV \to \CV $$
is the projection.

\medskip

Consider the diagram of operads
$$
\xymatrix{
E_{X \times \mathbb{R}_{>0}} \boxplus \on{Triv} \ar[r]^-{\beta}  & \Snglr_{X \times \mathbb{R}_{>0}} \boxplus \on{Triv} \ar[r]^-{\tilde{\Phi}} & \Snglr_{C(X)} \\
 & E_{\mathbb{R}_{>0}} \boxplus \on{Triv} \ar[u]^{\overset{\circ}{\pi}{}^{-1}}\ar[r]^-{\Phi_{\on{pt}}} & E_{\mathbb{R}_{\geq 0}}\ar[u]_{\pi^{-1}}
}
$$
Unraveling the definitions, we have that $q \circ \Cref{e:map of monads}$ is given by the composite
$$ \int_{\mathbb{R}_{\geq 0}} \circ \ 
\left((\Phi_{\on{pt}})_! \circ (\overset{\circ}{\pi}{}^{-1})^* \to (\pi^{-1})^* \circ \tilde{\Phi}_! \right)
\circ (\beta)_* ,$$
where the natural transformation is given by adjunction.  By \Cref{p:properties of operadic restr}\ref{i:operadic lke}, it suffices to show that the functor
\begin{equation}
\label{e:to show cofinal}
((E_{\mathbb{R}_{>0}} \boxplus \on{Triv})^{\otimes}_{\on{act}})_{/\mathbb{R}_{\geq 0}} \overset{\overset{\circ}{\pi}{}^{-1}}{\to}
((\Snglr_{X \times \mathbb{R}_{>0}} \boxplus \on{Triv})^{\otimes}_{\on{act}})_{/C(X)}
\end{equation}
is cofinal.  Now, $((E_{\mathbb{R}_{>0}} \boxplus \on{Triv})^{\otimes}_{\on{act}})_{/\mathbb{R}_{\geq 0}}$
has a cofinal subcategory with two objects and only identity morphisms given by $\{ [\mathbb{R}_{>0}], [\mathbb{R}_{>0}, t]\}$,
where $t \in \on{Triv}_{\langle 1 \rangle}$ is the unique object.  Moreover,
$$ \overset{\circ}{\pi}{}^{-1} (\{ [\mathbb{R}_{>0}], [\mathbb{R}_{>0}, t]\}) = \{ [X \times \mathbb{R}_{>0}], [X\times \mathbb{R}_{>0}, t]\} $$
is also a cofinal subcategory.  Hence, \Cref{e:to show cofinal} is cofinal as desired.
\qed

\ssec{Framed $E_2$-algebras}
In this subsection, we will give a more algebraic description of the category of algebras over the framed $E_2$-operad.

\sssec{}
Consider the natural map $BSO(2) \to BTop(2)$. Let $E_{SO(2)}$ be the corresponding operad as in
\cite[Definition 5.4.2.10]{lurieHA}. By \cite[Example 5.4.2.16]{lurieHA}, this is the \emph{framed $E_2$-operad}
$fE_2$, as studied in \cite{SW}.  By \cite[Remarks 5.4.2.13 and 2.3.3.4]{lurieHA}, we have that $E_{SO(2)}$ is the (homotopy) orbits of the natural $SO(2)$-action on the operad $E_2$.  It follows that we have a natural equivalence
\begin{equation}
\on{Alg}_{fE_2}(\CV) \simeq \on{Alg}_{E_2}(\CV)^{SO(2)},
\end{equation}
where the right hand side is the category of (homotopy) fixed points for the induced $SO(2)$-action.

\medskip

Before stating the main result of this section, we recall the relation between Hochschild homology and factorization homology, and the additivity theorem for factorization algebras.

\sssec{Hochschild homology}\label{sss:hochschild}
By \cite[Remark 5.4.5.2]{lurieHA}, we have an equivalence
\begin{equation}\label{e:so(2) on circle}
\on{Alg}_{S^1}(\CV)^{SO(2)} \simeq \on{Alg}_{E_1}(\CV)
\end{equation}
and by \cite[Theorem 5.5.3.11]{lurieHA}, the composite functor
$$ \mathbf{HH}_{*}: \on{Alg}_{E_1}(\CV) \simeq \on{Alg}_{S^1}(\CV)^{SO(2)} \overset{\int_{S^1}}{\longrightarrow} \CV^{SO(2)} $$
is given by Hochschild homology with a natural $SO(2)$-action (see \cite[Sect. 3.2]{AMR} for a combinatorial description of this circle action).

\sssec{Additivity}
Suppose that we have manifolds $M$ and $N$.  We have a bifunctor of operads (in the sense of \cite[Definition 2.2.5.3]{lurieHA})
$$E_M^{\otimes} \times E_N^{\otimes} \to E_{M\times N}^{\otimes} $$
given by product.  By \cite[Remark 5.4.2.14]{lurieHA}, this induces an equivalence of categories
\begin{equation}\label{e:additivity}
\on{Alg}_{M \times N}(\CV) \overset{\sim}{\longrightarrow} \on{Alg}_M(\on{Alg}_N(\CV)) .
\end{equation}

\sssec{}
We will be interested in how the isomorphism \Cref{e:additivity} interacts with factorization homology. We have a commutative diagram
$$
\xymatrix{
E_M^{\otimes} \times E_N^{\otimes} \ar[r]\ar[d] & E_{M \times N}^{\otimes} \ar[d] \\
E_M^{\otimes} \times \Snglr_N^{\otimes} \ar[r] & \Snglr_{M\times N}^{\otimes}
}
$$
where the vertical functors are inclusions and the horizontal functors are given by product of manifolds.  The vertical functors are bifunctors of operads in the sense of \cite[Definition 2.2.5.3]{lurieHA}.  This gives a commutative diagram
\begin{equation}\label{e:additivity fact}
\xymatrix{
\on{Alg}_{\Snglr_{M\times N}}(\CV) \ar[r]\ar[d] & \on{Alg}_{M}(\on{Alg}_{\Snglr_N}(\CV)) \ar[d]\\
\on{Alg}_{M\times N}(\CV)\ar[r] & \on{Alg}_M(\on{Alg}_N(\CV))
}
\end{equation}

\medskip
By \cite[Theorem 5.5.3.2]{lurieHA}, the operadic left Kan extension
\begin{equation}\label{e:fact homology symmetric monoidal}
\on{Alg}_N(\CV) \to \on{Alg}_{\Snglr_N}(\CV)
\end{equation}
is a symmetric monoidal functor.  It follows that the left adjoint to the restriction functor $\on{Alg}_{M}(\on{Alg}_{\Snglr_N}(\CV)) \to \on{Alg}_M(\on{Alg}_N(\CV))$ is given by
$$ \on{Alg}_M(\Cref{e:fact homology symmetric monoidal}):  \on{Alg}_M(\on{Alg}_N(\CV)) \to \on{Alg}_{M}(\on{Alg}_{\Snglr_N}(\CV)) .$$

\begin{prop}\label{p:additivity bc}
The diagram
$$ \xymatrix{
\on{Alg}_{M\times N}(\CV)\ar[r]\ar[d] & \on{Alg}_M(\on{Alg}_N(\CV)) \ar[d]^{\on{Alg}(\Cref{e:fact homology symmetric monoidal})} \\
\on{Alg}_{\Snglr_{M\times N}}(\CV) \ar[r] & \on{Alg}_{M}(\on{Alg}_{\Snglr_N}(\CV))
} $$
obtained from \Cref{e:additivity fact} by passing to left adjoints of the vertical functors naturally commutes.
\end{prop}
\begin{proof}
For every disk embedded in $M$, $D \in (E_M)_{\langle 1 \rangle}$, we have a commutative diagram
$$ \xymatrix{
\on{Alg}_M(\on{Alg}_N(\CV)) \ar[d]_{\on{Alg}_M(\Cref{e:fact homology symmetric monoidal})}\ar[r]^-{\int_D} & \on{Alg}_N(\CV)\ar[d]^{\Cref{e:fact homology symmetric monoidal}} \\
 \on{Alg}_{M}(\on{Alg}_{\Snglr_N}(\CV)) \ar[r]^-{\int_D} & \on{Alg}_{\Snglr_N}(\CV)
}. $$
The horizontal functors $\int_D$ are jointly conservative over all $D \in (E_M)_{\langle 1 \rangle}$.  Therefore, it suffices to show that the corresponding diagram
$$ \xymatrix{
\on{Alg}_{M\times N}(\CV)\ar[r]^-{\int_D }\ar[d] & \on{Alg}_N(\CV) \ar[d] \\
\on{Alg}_{\Snglr_{M\times N}}(\CV) \ar[r]^-{\int_D } & \on{Alg}_{\Snglr_N}(\CV)
} $$
commutes for every $D \in (E_M)_{\langle 1 \rangle}$.  This follows from \Cref{p:pre base change iso} and \Cref{p:base change}.
\end{proof}

\sssec{}
By the definition of relative factorization homology, we obtain:

\begin{cor}\label{c:additivity and fubini}
The composite functor
$$
\xymatrix{
\on{Alg}_{M \times N}(\CV) \ar[r]^-{\sim} & \on{Alg}_{M}(\on{Alg}_N(\CV)) \ar[rr]^-{\on{Alg}_M(\int_N)} & & \on{Alg}_M(\CV)
}$$
is canonically isomorphic to the functor of relative factorization homology along the projection $M \times N \to M$.
\end{cor}

\sssec{}
We are now ready to state the main theorem of this section.

\begin{thm}\label{t:framed E-2}
There is a fully faithful embedding
\begin{equation}\label{e:framed E-2 embedding}
\on{Alg}_{fE_2}(\CV) \hookrightarrow \on{Alg}_{E_2}(\CV) \underset{\on{Alg}(\CV^{SO(2)})}{\times} \on{AlgMod}_*(\CV^{SO(2)}),
\end{equation}
where the structure map $\on{Alg}_{E_2}(\CV) \to \on{Alg}(\CV^{SO(2)})$ is given by the composite
$$ \on{Alg}_{E_2}(\CV) \simeq \on{Alg}(\on{Alg})(\CV) \overset{\on{Alg}(\mathbf{HH}_{*})}{\longrightarrow} \on{Alg}(\CV^{SO(2)}) .$$
Moreover, the essential image of \Cref{e:framed E-2 embedding} is given by objects
$$ (A \in \on{Alg}_{E_2}(\CV), M \in \on{Mod}_{\hoch{A}}(\CV^{SO(2)})_{\hoch{A}/}) $$
such that the composite
$$ A \to \hoch{A} \to M $$
is an isomorphism.
\end{thm}

\sssec{}
The remainder of this subsection is devoted to the proof of \Cref{t:framed E-2}.

Let $\mathbb{R}^2_* \simeq C(S^1)$ denote the stratified manifold given by the plane with a zero dimensional stratum given by the origin.  By \cite[Lemma 2.23]{AFT} the restriction functor
\begin{equation}\label{e:add a stratification}
\on{Alg}_{E_2}(\CV) \simeq \on{Alg}_{\mathbb{R}^2}(\CV) \to \on{Alg}_{\mathbb{R}^2_*}(\CV)
\end{equation}
is fully faithful, with essential image given by $A \in \on{Alg}_{\mathbb{R}^2_*}(\CV)$ such that for any open embedding $\mathbb{R}^2 \hookrightarrow \mathbb{R}^2 - \{0\}$, the corresponding map
$$ A(\mathbb{R}^2) \to A(\mathbb{R}^2_*) $$
is an isomorphism.

\sssec{}
By \Cref{t:algebras on cone} and \Cref{sss:pointed modules}, we have an equivalence
\begin{equation}\label{e:decomp of cone on S-1}
\on{Alg}_{\mathbb{R}^2_*}(\CV) \simeq \on{Alg}_{\mathbb{R}^2-\{0\}}(\CV) \underset{\on{Alg}(\CV)}{\times} \on{AlgMod}_*(\CV)
\end{equation}
where the structure map $\on{Alg}_{\mathbb{R}^2-\{0\}}(\CV) \to \on{Alg}(\CV)$ is given by relative factorization homology along the norm map
$$ \on{Nm}: \mathbb{R}^2-\{0\} \to \mathbb{R}_{>0} .$$
Moreover, the essential image of the composite $\Cref{e:decomp of cone on S-1} \circ\Cref{e:add a stratification}$ is given by objects
$$ (A \in \on{Alg}_{\mathbb{R}^2-\{0\}}(\CV), M \in \on{Nm}_*(A)\on{-mod}_{\on{Nm}_*(A)/}) $$
such that for any disk $D \subset \mathbb{R}^2-\{0\}$, the composite
$$ \int_D A \to \int_{\mathbb{R}^2-\{0\}} A \simeq \on{Nm}_*(A) \to M $$
is an isomorphism.

\sssec{}
Both \Cref{e:decomp of cone on S-1} and \Cref{e:add a stratification} are clearly $SO(2)$-equivariant.  It follows that we have a fully faithful embedding
$$ \on{Alg}_{fE_2}(\CV) \simeq \on{Alg}_{E_2}(\CV)^{SO(2)} \hookrightarrow \on{Alg}_{\mathbb{R}^2-\{0\}}(\CV)^{SO(2)} \underset{\on{Alg}(\CV^{SO(2)})}{\times} \on{AlgMod}_*(\CV^{SO(2)}) .$$
The desired result now follows from the following identification:

\begin{prop}
There is an equivalence
$$ \on{Alg}_{\mathbb{R}^2-\{0\}}(\CV)^{SO(2)} \simeq \on{Alg}_{E_2}(\CV) $$
and the composite functor
$$ \on{Alg}_{E_2}(\CV) \simeq \on{Alg}_{\mathbb{R}^2-\{0\}}(\CV)^{SO(2)} \overset{\on{Nm}_*}{\longrightarrow} \on{Alg}_{\mathbb{R}_{>0}}(\CV)^{SO(2)} \simeq \on{Alg}_{E_1}(\CV^{SO(2)}) $$
is isomorphic to
$$ \on{Alg}_{E_2}(\CV) \simeq \on{Alg}_{E_1}(\on{Alg}_{E_1}(\CV)) \overset{\on{Alg}_{E_1}(\mathbf{HH}_{*})}{\to} \on{Alg}_{E_1}(\CV^{SO(2)}) .$$
\end{prop}
\begin{proof}
We have, using polar coordinates, $\mathbb{R}^2 - \{0\} \simeq \mathbb{R}_{>0} \times S^1$ and the $SO(2)$-action is on the second factor.  Thus, by \Cref{e:additivity}, we have
$$ \on{Alg}_{\mathbb{R}^2-\{0\}}(\CV) \simeq \on{Alg}_{\mathbb{R}_{>0}}(\on{Alg}_{S^1}(\CV)) .$$
Hence,
$$ \on{Alg}_{\mathbb{R}^2-\{0\}}(\CV)^{SO(2)} \simeq \on{Alg}_{\mathbb{R}_{>0}}(\on{Alg}_{S^1}(\CV))^{SO(2)} \simeq \on{Alg}_{\mathbb{R}_{>0}}(\on{Alg}_{S^1}(\CV)^{SO(2)}) \simeq \on{Alg}_{\mathbb{R}_{>0}}(\on{Alg}_{E_1}(\CV)), $$
by \Cref{e:so(2) on circle}.

Moreover, by \Cref{c:additivity and fubini}, the functor
$$ \on{Alg}_{\mathbb{R}^2-\{0\}}(\CV) \overset{\on{Nm}_*}{\longrightarrow} \on{Alg}_{\mathbb{R}_{>0}}(\CV) $$
is isomorphic to the composite
$$ \on{Alg}_{\mathbb{R}^2-\{0\}}(\CV) \simeq \on{Alg}_{\mathbb{R}_{>0}}(\on{Alg}_{S^1}(\CV)) \overset{\on{Alg}_{\mathbb{R}_{>0}}(\int_{S^1})}{\longrightarrow} \on{Alg}_{\mathbb{R}_{>0}}(\CV) .$$
The desired result now follows by \Cref{sss:hochschild}.
\end{proof}

\section{Noncommutative orientations and the relative cyclic Deligne conjecture}\label{cycdel}

In this section, we explain how \Cref{t:framed E-2} implies our main theorem, the relative cyclic Deligne conjecture for 
(a generalization of) ``relative Calabi-Yau structures'' in the sense of \cite{bravdyck}.

\ssec{Non-commutative orientations}

\sssec{}

%
Given a $2$-algebra $R$ and a non-commutative calculus of $\hoch{R}$ acting circle-equivariantly on $M$, we obtain an action of $R$ on $M$. Namely, $\hoch{R}$ can be computed as factorization homology of $R$ over a circle, while $R$ itself can be computed as factorization homology of $R$ over an open interval. Embedding an open interval into a circle, we obtain a map of factorization homologies 
\[
\iota: R \arr \hoch{R}
\]
compatible with the residual associative algebra structures (but breaking the circle symmetry in the target). 

Given a ``vector'' $X \in M$ (that is, a point $1_{\Abold} \stackrel{x}\arr M$), we denote the composite 
\[
R \stackrel{1_{R} \otimes x}\arr R \otimes M \stackrel{\iota \otimes id}\arr \hoch{R} \otimes M \stackrel{act}\arr M
\]
simply as
\begin{equation}\label{capprod}
R \stackrel{- \cap x}\arr M.
\end{equation}

\sssec{}

Let $\Cbold$ be an $\Abold$-linear dualizable category. A {\it non-commutative orientation} on $\Cbold$ of dimension $d$ is by definition a circle-invariant class $\theta \in \circinv{\hoch{\Cbold}}[-d]$ such that capping against the underlying vector $\theta$ 
\[
\hochco{\Cbold} \stackrel{- \cap \theta}\arr \hoch{\Cbold}[-d]
\]
is an isomorphism. 

More generally, a {\it non-commutative relative orientation} of dimension $d$ on a dualizable functor $f: \Cbold \arr \Dbold$ is a circle invariant class $\theta \in \circinv{\hoch{\Cbold \stackrel{f}\arr \Dbold}}[-d]$ such that capping on $\theta$ 
\[
\hochco{\Cbold \stackrel{f}\arr \Dbold} \stackrel{- \cap \theta}\arr \hoch{\Cbold \stackrel{f}\arr \Dbold}[-d]
\]
is an equivalence. Note that taking $\Dbold=0$, we are reduced to the absolute case. 

\sssec{}

Circle-invariants of Hochschild homology are often referred to as ``negative cyclic homology'', so for $\Abold=\Vect$, one would write the datum of an orientation as $\theta \in \negcyc{\Cbold}[-d] := \hoch{\Cbold}^{hS^{1}}[-d]$, while for $\Abold=\Spct$, one would write $\theta \in TC^{-}(\Cbold)[-d]:=THH(\Cbold)^{hS^{1}}[-d]$, and likewise for relative classes.

\begin{rem}
In the case that $A=\Spct$, $TC^{-}$ admits a refinement given by \emph{topological cyclic homology} $TC$. One can consider a refined version of an orientation $\theta \in TC(\Cbold)[-d]$ (whose image under the natural map $TC(\Cbold) \to TC^{-}(\Cbold)$ is an orientation as above).  It would be interesting to study the consequences of having such a refinement.
\end{rem}

\sssec{}

In practice, it often happens that the non-commutative relative orientation $\theta \in  \circinv{\hoch{\Cbold \stackrel{f}\arr \Dbold}}[-d]$ restricts along the boundary to a non-commutative absolute orientation $\overline{\theta} \in \circinv{\hoch{\Cbold}}[1-d]$. We say in this case that $\theta$ is {\it non-degenerate on the boundary}. 

Note that if $\theta$ is non-degenerate on the boundary, then action on $\theta$ and $\overline{\theta}$ induces an equivalence of fiber sequences
\begin{equation}\label{nclefschetzseq}
\xymatrix{\Homm{\Abold}{{\rm cof}(\varepsilon)}{\Id{\Dbold}} \ar[r] \ar[d]_{\simeq} & \hochco{\Cbold \stackrel{f}\arr \Dbold} \ar[d]^{- \cap \theta}_{\simeq} \ar[r] & \hochco{\Cbold} \ar[d]^{- \cap \overline{\theta}}_{\simeq} \\
\hoch{\Dbold}[-d] \ar[r] & \hoch{\Cbold \stackrel{f}\arr \Dbold}[-d] \ar[r] & \hoch{\Cbold}[1-d].}
\end{equation}

Here, ${\rm cof}(\varepsilon)$ is defined by the fiber sequence of functors $ff^{r} \stackrel{\varepsilon}\arr \Id{\Dbold} \arr {\rm cof}(\varepsilon)$. The identification of the fiber of $\hochco{\Cbold \stackrel{f}\arr \Dbold} \arr \hochco{\Cbold}$ with $\Homm{\Abold}{{\rm cof}(\varepsilon)}{\Id{\Dbold}}$ is readily verified using the isomorphism $\hochco{\Cbold \stackrel{f}\arr \Dbold} \simeq \hochco{\Cbold} \times_{\End{f}} \hochco{\Dbold} $.

\sssec{}

Here is our main theorem, the relative cyclic Deligne conjecture.

\begin{thm}[Relative cyclic Deligne]\label{relcycdel} Let $\Abold$ be a presentably symmetric monoidal  category and 
$f: \Cbold \arr \Dbold$ an $\Abold$-linear dualizable functor with non-commutative relative orientation $\theta \in \circinv{\hoch{\Cbold \stackrel{f}\arr \Dbold}}[-d]$ of dimension $d$. Then the relative Hochschild cohomology $\hochco{ \Cbold \stackrel{f}\arr \Dbold}$ has an induced structure of framed $E_{2}$-algebra. Equivalently, the shifted relative Hochschild homology $\hoch{\Cbold \stackrel{f}\arr \Dbold}[-d]$ has an induced structure of framed $E_{2}$-algebra. In particular, if $\Cbold$ is an $\Abold$-linear dualizable category with absolute non-commutative orientation, then its Hochschild cohomology $\hochco{\Cbold}$ carries an induced structure of framed $E_{2}$-algebra.

Moreover, if the non-commutative relative orientation $\theta \in \circinv{\hoch{\Cbold \stackrel{f}\arr \Dbold}}[-d]$ is non-degenerate on the boundary, so that $\overline{\theta} \in \circinv{\hoch{\Cbold}}[1-d]$ is an absolute non-commutative orientation, then the map
$$\hochco{\Cbold \stackrel{f}\arr \Dbold} \arr \hochco{\Cbold}$$
is a map of framed $E_{2}$-algebras, hence the fiber
$$\fib{\hochco{\Cbold \stackrel{f}\arr \Dbold}}{\hochco{\Cbold}}$$
carries the structure of non-unital framed $E_{2}$-algebra. Equivalently, via the isomorphism \Cref{nclefschetzseq} of fiber sequences, the sequence
$$\hoch{\Dbold}[-d] \arr \hoch{\Cbold \stackrel{f}\arr \Dbold}[-d] \arr \hoch{\Cbold}[1-d]$$ carries the structure of a fiber sequence of framed $E_{2}$-algebras, with the fiber non-unital.
\end{thm}

\begin{proof} Given what we have already established, the proof is straightforward. Namely, by
\Cref{t:framed E-2}, to give a framed $E_{2}$-algebra is to give an underlying $E_{2}$-algebra $R$, a circle-module $M$ for $\hoch{R}$, and a circle-invariant vector $m \in M^{hS^{1}}$ such that capping $R \stackrel{- \cap v}\arr M$ is an isomorphism. Taking $R=\hochco{\Cbold \stackrel{f}\arr \Dbold}$, $M=\hoch{\Cbold \stackrel{f}\arr \Dbold}$, and $m=\theta \in \hoch{\Cbold \arr \Dbold}^{hS^{1}}[-d]$, we obtain the desired framed $E_{2}$-algebra structure on $\hoch{\Cbold \stackrel{f}\arr \Dbold}$. 

The remaining statements follow by naturality of the constructions. 

\end{proof}

\sssec{} As we shall see in the next subsection, our examples of relative non-commutative orientations are induced by relative left Calabi-Yau structures in the sense of \cite{bravdyck},\cite{bravdyck2}. For an alternative introduction to relative Calabi-Yau structures, see \cite{kellerwang}. We now briefly review such structures, following the treatment in \cite{bravdyck2} and explain how they induce relative non-commutative orientations. 

\sssec{}

Recall that a dualizable $\Abold$-linear category $\Cbold$ is {\it smooth} over $\Abold$ if the evaluation functor
$$\ev{\Cbold} : \Cbold^{\vee} \otimes_{\Abold} \Cbold \arr \Abold$$
has an $\Abold$-linear left adjoint
$$\evL{\Cbold} :\Abold \arr \Cbold^{\vee} \otimes_{\Abold} \Cbold.$$
If $\Cbold$ is smooth, then the {\it inverse dualizing functor} is the image of the tensor unit under the left adjoint to evaluation:
$$\Inv{\Cbold}:= \evL{\Cbold} \in \Cbold^{\vee} \otimes_{\Abold} \Cbold \simeq \Endd{\LinCat}{\Cbold}.$$
The $\Abold$-linear trace is then corepresented by $\Inv{\Cbold}$, so that in particular there is an identification
$$\hoch{\Cbold} \simeq \Homm{\Abold}{\Inv{\Cbold}}{\Id{\Cbold}}.$$

\sssec{}

Let $f: \Cbold \arr \Dbold$ be a dualizable functor between smooth categories and consider a class
$$\theta \in \hoch{\Cbold \stackrel{f}\arr \Dbold}[-d]$$
and its image $\overline{\theta} \in \hoch{\Cbold}[1-d]$ under the boundary map
$$\hoch{\Cbold \stackrel{f}\arr \Dbold}[-d] \arr \hoch{\Cbold}[1-d].$$
Since $\overline{\theta}$ is given as the image of $\theta$, the image of $\overline{\theta}$ along $\hoch{\Cbold}[1-d] \arr \hoch{\Dbold}[1-d]$ is equipped with a null-homotopy. 

With respect to the identifications
$$\hoch{\Cbold} \simeq \Homm{\Abold}{\Inv{\Cbold}}{\Id{\Cbold}} \mbox{ and } \hoch{\Dbold} \simeq \Homm{\Abold}{\Inv{\Dbold}}{\Id{\Dbold}},$$
the image of $\Inv{\Cbold}[1-d] \stackrel{\overline{\theta}} \arr \Id{\Cbold}$ under the induced map $\hoch{\Cbold}[1-d] \arr \hoch{\Dbold}[1-d]$ is given by the composite 
\[
\Inv{\D}[d-1] \arr f \Inv{\Cbold}[d-1] f^{r} \stackrel{f\alpha f^{r}}\arr ff^{r} \arr \Id{\D}.
\]
Since the composite is equipped with a null-homotopy, there is an induced commutative diagram of endofunctors of $\Dbold$
\begin{equation}\label{leftrelCYdiag}
\xymatrix{\Inv{\Dbold}[d-1] \ar[r] \ar[d] & f \Inv{\Cbold} f^{r} [d-1] \ar[r] \ar[d]^{f \overline{\theta} f^{r}} & {\rm cof} \ar[d] \\
{\rm fib} \ar[r] & ff^{r} \ar[r]^{\varepsilon} & \Id{\Dbold}}.
\end{equation}

By definition, a {\it relative left Calabi-Yau structure} of dimension $d$ on $\Cbold \stackrel{f}\arr \Dbold$ is a circle-invariant class $\theta \in \hoch{\Cbold \stackrel{f}\arr \Dbold}^{hS^{1}}[-d]$ such that all vertical arrows in the arrow diagram \Cref{leftrelCYdiag} are isomorphisms. In particular, the map 
$$f \Inv{\Cbold} f^{r}[d-1] \stackrel{f \overline{\theta} f^{r}}\arr ff^{r}$$
is required to be an isomorphism. In practice, the map $\Inv{\Cbold}[d-1] \stackrel{{\theta}}\arr \Id{\Cbold}$ is usually already an isomorphism, so that the boundary category $\Cbold$ is equipped with an {\it absolute left Calabi-Yau structure} of dimension $d-1$.
 
 \sssec{} We formulate the relation between relative orientations and relative left Calabi-Yau structures.
 
 \begin{lem}\label{relcytorelor}
Suppose that we are given a relative left Calabi-Yau structure $\theta \in \hoch{\Cbold \stackrel{f}\arr \Dbold}[-d]$ of dimension $d$ on an $\Abold$-linear dualizable functor $\Cbold \stackrel{f}\arr \Dbold$ between smooth categories, and suppose that the induced boundary class $\overline{\theta} \in \hoch{\Cbold}[1-d]$ is an absolute left Calabi-Yau structure on $\Cbold$. Then capping on $\theta \in \hoch{\Cbold \stackrel{f}\arr \Dbold}[-d]$ gives an isomorphism
$$\hochco{\Cbold \stackrel{f}\arr \Dbold} \stackrel{- \cap \theta}\simeq \hoch{\Cbold \stackrel{f}\arr \Dbold}[-d].$$
In short, a relative left Calabi-Yau structure induces a non-commutative relative orientation.
\end{lem}
 
\begin{proof} 
 
 Capping against $\overline{\theta}$ gives a map $\hochco{\Cbold} \stackrel{- \cap \overline{\theta}}\arr \hoch{\Cbold}[1-d]$, which is evidently an isomorphism since it is given by applying $\Homm{\Abold}{-}{\Id{\Cbold}}$ to the isomorphism $\Inv{\Cbold}[d-1] \stackrel{\overline{\theta}}\arr \Id{\Cbold}$. Then capping against $\theta$ gives a commutative diagram of fiber sequences
 \begin{equation}
 \xymatrix{\Homm{\Abold}{{\rm cof}(\varepsilon)}{\Id{\Dbold}} \ar[r] \ar[d] & \hochco{\Cbold \stackrel{f}\arr \Dbold} \ar[r]\ar[d] & \hochco{\Cbold} \ar[d]^{\simeq} \\
 \hoch{\Dbold}[-d] \ar[r] & \hoch{\Cbold \stackrel{f}\arr \Dbold}[-d] \ar[r] & \hoch{\Cbold}[1-d].}
 \end{equation}
 We claim that the left vertical arrow is an isomorphism, hence the middle vertical arrow is as well, and we obtain a non-commutative relative orientation as claimed. Indeed, under the identification
$$\hoch{\Dbold}[d] \simeq \Hom{\Abold}{\Inv{\Dbold}[d]}{\Id{\Dbold}},$$
the left hand vertical arrow is obtained by applying $\Homm{\Abold}{-}{\Id{\Dbold}}$ to the arrow $\Inv{\Dbold}[d] \arr {\rm cof}(\varepsilon)$, which is an isomorphism by the defining diagram \Cref{leftrelCYdiag} of a relative Calabi-Yau structure. 
  
\end{proof}







\ssec{Examples of non-commutative orientations}

We give some examples of relative left Calabi-Yau structures, hence by \Cref{relcytorelor} examples of non-commutative relative orientations. Since we provide references to the details, we shall be brief.

\sssec{String topology}

We fix a ground commutative ring spectrum $E$ and let ${\bf H}_{*}(X)=E \wedge X_{+}$ and ${\bf H}^{*}(M)=\Homm{\Spct}{\Sigma^{\infty}X_{+}}{E}$. Thus we are using (co)homology notation for the relevant spectra and suppressing the dependence on $E$. Likewise, we let $\Loc{X}$ denote the category of local systems of $E$-module spectra. 

First, consider an oriented $d$-manifold with boundary $(M,\partial M)$ with relative orientation class $[M,\partial M] \in {\bf H}_{*}(M,\partial M)[-d]$ in the sense of Lefschetz duality, namely capping on $[M,\partial M]$ induces an isomorphism of fiber sequences
\[
\xymatrix{{\bf H}^{*}(M,\partial M) \ar[r] \ar[d]_{\simeq} & {\bf H}^{*}(M) \ar[r] \ar[d]_{\simeq}^{{- \cap [M,\partial M] }} & {\bf H}^{*}(\partial M) \ar[d]_{\simeq} \\
{\bf H}_{*}(\partial M)[-d] \ar[r] & {\bf H}_{*}(M,\partial M)[-d] \ar[r] & {\bf H}_{*}(\partial M)[1-d]}
\]
We claim that the homological pushforward/induction functor
\[
i_{!} : \Loc{\partial M} \arr \Loc{M}
\]
carries an induced non-commutative relative orientation of dimension $d$. We sketch the construction, referring to \cite[Section 5.1]{bravdyck} for more details.

Pushing forward along the ``constant loops'' vertical arrows in the circle-invariant diagram  
\[
\xymatrix{\partial M \ar[r] \ar[d] & M \ar[d] \\
L\partial M \ar[r] & LM}
\] 
we obtain a circle-invariant class $[M,\partial M] \in {\bf H}_{*}(LM,L\partial M)^{hS^{1}}[-d]$. By the Goodwillie-Jones theorem, there is a circle-equivariant isomorphism ${\bf H}_{*}(LM,L\partial M) \simeq \hoch{\Loc{M},\Loc{\partial M}}$, hence
$$[M,\partial M] \in {\bf H}_{*}(LM,L\partial M)^{hS^{1}}[-d]$$
gives a class $\theta \in \hoch{\Loc{M},\Loc{\partial M}}^{hS^{1}}[-d]$. Using Lefschetz duality, one deduces that this class gives a relative left Calabi-Yau structure that is non-degenerate on the boundary, and hence a non-commutative orientation that is non-degenerate on the boundary.

Applying \Cref{relcycdel}, we immediately obtain chain-level genus zero string topology operations on relative loop homology.

\begin{cor}\label{stringtopcor}
Fix a ground commutative ring spectrum $E$ and let $[M,\partial M] \in {\bf H}_{*}(M,\partial M)[-d]$ be a relative orientation of dimension $d$. Then there is an induced fiber sequence of (chain-level) framed $E_{2}$-algebra structures 
\[
{\bf H}_{*}(LM)[-d] \arr {\bf H}_{*}(LM,L\partial M)[-d] \arr {\bf H}_{*}(L\partial M)[1-d].
\]
in which the fiber is a non-unital framed $E_{2}$-algebra.
\end{cor}

\sssec{Anticanonical divisors}\label{ss:anticanon indcoh} 

The algebro-geometric analogue of an oriented manifold with boundary is a scheme with anti-canonical divisor, with the categories of local systems on the manifolds replaced by the categories of ind-coherent sheaves on the schemes.

Fix an integral Gorenstein variety $X$ of dimension $d$ over a perfect field $k$, and choose a non-zero anti-canonical section $\sigma \in K^{-1}_{X}$. Then the choice of anti-canonical section $\sigma$ determines a trivialization of the canonical sheaf of the zero scheme $i: Z=Z(\sigma) \hookrightarrow X$. We claim that there is a natural non-commutative relative orientation on the functor $i_{*} : \IndCoh{Z} \arr \IndCoh{X}$. We briefly sketch the construction, referring to \cite[Section 5.2]{bravdyck} for more details.

By a Hochschild vanishing argument (\cite[Lemma 5.11]{bravdyck}), the circle-invariance comes for free, so the essential point in constructing the non-commutative relative orientation is to construct the corresponding relative Hochschild class. For that, one first uses natural identifications $\hoch{\IndCoh{Z}} \simeq \Hom{\IndCoh{Z}}{\Delta_{*}\O_{Z}}{\Delta_{*}\omega_{Z}}$ and $\hoch{\IndCoh{X}} \simeq \Hom{\IndCoh{X}}{\Delta_{*}\O_{X}}{\Delta_{*}\omega_{X}}$. 

The choice of an anti-canonical section $\sigma \in K^{-1}_{X}$ cutting out $Z$ then determines an isomorphism
$$\O_{Z} \simeq K_{Z} \simeq \omega_{Z}[1-d],$$
hence after pushing forward along the diagonal, gives a class $\hoch{\IndCoh{Z}}[1-d]$, giving a non-commutative orientation of dimension $d-1$ on $\IndCoh{Z}$. 

Moreover, the induced map in Hochschild homology $\hoch{\IndCoh{Z}} \arr \hoch{\IndCoh{X}}$ can be shown to identify with the map $\Hom{\IndCoh{Z}}{\Delta_{*}\O_{Z}}{\Delta_{*}\omega_{Z}} \arr \Hom{\IndCoh{X}}{\Delta_{*}\O_{X}}{\Delta_{*}\omega_{X}}$ given by sending
a map $\Delta_{*}\O_{Z} \arr \Delta_{*}\omega_{Z}[k]$ to the composite
\[
\Delta_{*}\O_{X} \arr \Delta_{*}i_{*}\O_{Z} \simeq (i \times i)_{*}\Delta_{*}\O_{Z} \arr (i \times i)_{*}\Delta_{*}\omega_{Z}[k] \simeq \Delta_{*}i_{*}\omega_{Z}[k] \simeq \Delta_{*}i_{*}i^{!}\omega_{X}[k] \arr \Delta_{*}\omega_{X}[k].
\]
In these terms, a lift along $\hoch{\IndCoh{X},\IndCoh{Z}}[-d] \arr \hoch{\IndCoh{Z}}[1-d]$ of the absolute orientation of $\IndCoh{Z}$ amounts to providing a null-homotopy of the composite
\[
\Delta_{*}\O_{X} \arr \Delta_{*}i_{*}\O_{Z} \simeq \Delta_{*}i_{*}\omega_{Z}[1-d] \arr \Delta_{*}\omega_{X}[1-d].
\]
One can show that such a null-homotopy is naturally determined by the choice of the anti-canonical section $\sigma$ and is moreover non-degenerate in the appropriate sense, thus giving a non-commutative relative orientation of dimension $d$ on the  $i_{*} : \IndCoh{Z} \arr \IndCoh{X}$.

\medskip

Applying \Cref{relcycdel}, we obtain the following.

\begin{cor}
Let $X$ be an irreducible Gorenstein variety over a perfect field and $Z=Z(\sigma)$ the zero scheme of a non-zero anti-canonical section $\sigma \in K^{-1}_{X}$. Then there is a fiber sequence of framed $E_{2}$-algebra structures
\[
\hoch{\IndCoh{X}}[-d] \arr \hoch{\IndCoh{Z} \stackrel{i_{*}}\arr \IndCoh{X}}[-d] \arr \hoch{\IndCoh{Z}}[1-d]
\]
in which the fiber is a non-unital framed $E_{2}$-algebra.
\end{cor}

\sssec{Doubled quivers and relative Calabi-Yau completions}

Relative orientations also appear surprisingly often in representation theory, often in fact from a universal construction of the ``relative Calabi-Yau completion'' of \cite{yeung}, a construction which takes a dualizable functor $\Cbold \arr \Dbold$ between smooth categories and produces a new dualizable functor $\widetilde{\Cbold} \arr \widetilde{\Dbold}$ equipped with a relative orientation of dimension $d$. 

In examples, the initial functor $\Cbold \arr \Dbold$ is often given as induction of modules along a map of algebras $R \arr S$, in which case the Calabi-Yau completion $\widetilde{\Cbold} \arr \widetilde{\Dbold}$ is given as induction of modules along a new map of algebras $\widetilde{R} \arr \widetilde{S}$.

As a particular case of the construction, start with a finite quiver $Q$ with vertex set $Q_{0}$ and arrow set $Q_{1}$, as well as the doubled quiver $\overline{Q}$, which in addition to all of the arrows in the original quiver $Q$ has for every arrow $a : i \arr j$ from $Q$ an additional dual arrow $a^{*} : j \arr i$ in the opposite direction. We shall also consider $Q_0$ itself as a quiver having only vertices and no non-trivial arrows. 

Consider now the map of path algebras $k[Q_0] \arr k[Q]$ induced by the inclusion of quivers $Q_0 \subset Q$. Then at the level of algebras the relative Calabi-Yau completion of dimension $2$ is given by a map $k[Q_{0}][t] \arr k[\overline{Q}]$ taking the element $t$ to the ``preprojective element'' $\sum_{} [a,a^{*}] \in k[\overline{Q}]$. 

For details on this particular example, see \cite{kellerwang}. 

\ssec{Non-commutative co-orientations}

\sssec{}

Let $\Dbold$ be a dualizable $\Abold$-linear category. A {\it non-commutative co-orientation} on $\Dbold$ of dimension $d$ is by definition a circle-invariant map $\tau : \hoch{\Dbold} \arr 1_{\Abold}[-d]$, that is, a circle invariant class 
\[
\tau \in \Homm{\Abold}{\hoch{\Dbold}}{1_{\Abold}}^{hS^{1}}[-d]
\]
such that capping
\[
\hochco{\Dbold} \stackrel{- \cap \tau}\arr \Homm{\Abold}{\hoch{\Dbold}}{1_{\Abold}}^{\vee}[-d]
\]
is an isomorphism. Here and below we are using $(-)^\vee$ as shorthand for $\Abold$-linear duality $\Homm{\Abold}{-}{1_{\Abold}}$.

More generally, {\it a non-commutative relative co-orientation} of dimension $d$ on a dualizable functor $f: \Dbold \stackrel{f}\arr \Cbold$ is a circle-invariant class 
\[
\tau \in \Homm{\Abold}{\hoch{\Dbold \stackrel{f}\arr \Cbold}}{1_{\Abold}}^{hS^{1}}[1-d]
\] 
such that capping
\[
\hochco{\Dbold \stackrel{f}\arr \Cbold} \stackrel{- \cap \tau}\arr \hoch{\Dbold \stackrel{f}\arr \Cbold}^{\vee}[1-d]
\]
is an isomorphism. 
Note the shift by $1-d$ rather than $-d$, which reflects the fact that the functor $\Dbold \arr \Cbold$ plays the role of ``restriction to the boundary''. In particular, if $\Cbold=0$, so that we have ``empty boundary", then a relative right Calabi-Yau structure of dimension $d$ is nothing but an absolute right Calabi-Yau structure on $\Dbold$. 

\sssec{}

In examples, it typically happens that restriction of a relative right Calabi-Yau structure
$$\tau : \hoch{\Dbold \arr \Cbold} \arr 1_{\Abold}[1-d]$$
to a circle-invariant map $\overline{\tau} : \hoch{\Cbold} \arr \hoch{\Dbold \stackrel{f}\arr \Cbold} \arr 1_{\Abold}[1-d]$ gives an absolute right Calabi-Yau structure on $\Cbold$ of dimension $d-1$. In this case, capping on $\tau$ and $\overline{\tau}$ gives an identification of fiber sequences
\[
\xymatrix{
\hochco{\Dbold;{\rm fib}(\eta)} \ar[r] \ar[d]_{\simeq} & \hochco{\Dbold \arr \Cbold} \ar[r]
 \ar[d]_{\simeq}^{- \cap \tau} & \hochco{\Cbold} \ar[d]_{\simeq}^{- \cap \overline{\tau}} \\
\hoch{\Dbold}^{\vee}[-d] \ar[r] & \hoch{\Dbold \arr \Cbold}^{\vee}[1-d] \ar[r] & \hoch{\Cbold}^{\vee}[1-d],}
\]
where $\eta$ is the unit $\Id{\Dbold} \stackrel{\eta}\arr f^{r}f$. The identification of the fiber of 
$\hochco{\Dbold \arr \Cbold} \stackrel{- \cap \tau}\arr  \hochco{\Cbold}$ with
$\hochco{\Dbold; {\rm fib}(\eta)}$ follows easily from the isomorphism
$$\hochco{\Dbold \arr \Cbold} \simeq \hochco{\Cbold} \times_{\End{f}} \hochco{\Dbold}.$$ 

\sssec{} Here is a variation on the cyclic Deligne conjecture, this time for relative co-orientations. 

\begin{thm}[Relative cyclic Deligne for co-orientations]\label{relcycdelprop}
Given a dualizable functor $f: \Dbold \arr \Cbold$ with non-commutative co-orientation $\tau \in \Homm{\Abold}{\hoch{\Dbold \stackrel{f}\arr \Cbold}}{1_{\Abold}}^{hS^{1}}[1-d]$ of dimension $d$, the relative Hochschild cohomology $\hochco{\Dbold \stackrel{f}\arr \Cbold}$ has an induced structure of framed $E_{2}$-algebra. 

If in addition $\tau : \hoch{\Dbold \stackrel{f}\arr \Cbold} \arr 1_{\Abold}[1-d]$ restricts to an absolute non-commutative co-orientation $\overline{\tau} : \hoch{\Cbold} \arr 1_{\Abold}[1-d]$, then there is a fiber sequence of framed $E_{2}$-algebras
$$\hochco{\Dbold;{\rm fib}(\eta)} \arr \hochco{\Dbold \arr \Cbold} \arr \hochco{\Cbold},$$
equivalently a fiber sequence of framed $E_{2}$-algebra structures 
$$\hoch{\Dbold}^{\vee}[-d] \arr  \hoch{\Dbold \stackrel{f}\arr \Dbold}^{\vee}[1-d] \arr  \hoch{\Cbold}^{\vee}[1-d].$$ Note that in this case, the framed $E_{2}$-algebra structure on the fiber $\hochco{\D;{\rm fib}(\eta)} \simeq \hoch{\Dbold}^{\vee}[-d]$ is non-unital.
\end{thm}

The proof is immediate and essentially the same as that of \Cref{relcycdel}.

\sssec{}

Examples of relative non-commutative co-orientations typically are induced from relative right Calabi-Yau structures in the sense of \cite{bravdyck}. We briefly describe the relation.

Recall that a dualizable $\Abold$-linear category $\Cbold$ is {\it proper} over $\Abold$ if the evaluation functor
$$\ev{\Cbold} : \Cbold^{\vee} \otimes_{\Abold} \Cbold \arr \Abold$$
has an $\Abold$-linear right adjoint
$$\evR{\Cbold} :\Abold \arr \Cbold^{\vee} \otimes_{\Abold} \Cbold.$$
If $\Cbold$ is proper, then the {\it Serre functor} is by definition the image of the tensor unit $1_{\Abold}$ under the right adjoint to evaluation:
$$\Serre{\Cbold}:= \evR{\Cbold} \in \Cbold^{\vee} \otimes_{\Abold} \Cbold \simeq \Endd{\Abold}{\Cbold}.$$ 
The $\Abold$-linear dual of $\ev{\Cbold}$ is then represented by $\evR{\Cbold}$, so that in particular there is an identification
$$\hoch{\Cbold}^{\vee} \simeq \Homm{\Endd{\Abold}{\Cbold}}{\Id{\Cbold}}{\Serre{\Cbold}}.$$

\sssec{}
Let $f: \Dbold \arr \Cbold$ be a dualizable functor between proper categories and consider a class
$$\tau :\hoch{\Dbold \stackrel{f}\arr \Cbold} \arr 1_{\Abold}[1-d]$$
and its restriction to a class $\overline{\tau} : \hoch{\Cbold} \arr 1_{\Abold}[1-d]$. Since $\overline{\tau}$ is given as the restriction of $\tau$, the further restriction of $\overline{\tau}$ along $\hoch{\Dbold}[d-1] \arr \hoch{\Cbold}[d-1]$ is equipped with a null-homotopy.

With respect to the identifications
$$\hoch{\Cbold}^{\vee} \simeq \Homm{\Abold}{\Id{\Cbold}}{\Serre{\Cbold}} \mbox{ and } \hoch{\Dbold}^{\vee} \simeq \Homm{\Abold}{\Id{\Dbold}}{\Serre{\Dbold}},$$
the restriction of $\Id{\Cbold} \stackrel{\overline{\tau}} \arr \Serre{\Cbold}[1-d]$ under the induced map $\hoch{\Cbold}^{\vee}[1-d] \arr \hoch{\Dbold}^{\vee}[1-d]$ is given by the composite 
\[
\Id{\D} \arr f^{r}f \stackrel{f^{r}\overline{\tau}f}\arr f^{r}\Serre{\Cbold}[1-d] f \arr \Serre{\Dbold}[1-d].
\]
Since the composite is equipped with a null-homotopy, there is an induced commutative diagram of endofunctors of $\Dbold$
\begin{equation}\label{rightrelCYdiag}
\xymatrix{\Id{\Dbold} \ar[r]^{\eta} \ar[d] & f^{r}f \ar[r] \ar[d]^{f^{r} \overline{\tau} f} & {\rm cof} \ar[d] \\
{\rm fib} \ar[r] & f^{r}\Serre{\Cbold}f^{r}[1-d] \ar[r] & \Serre{\Dbold}[1-d]}.
\end{equation}

By definition, a {\it relative right Calabi-Yau structure} of dimension $d$ on $\Dbold \stackrel{f}\arr \Cbold$ is a class 
\[
\tau \in \Homm{\Abold}{\hoch{\Dbold \stackrel{f}\arr \Cbold}}{1_{\Abold}}^{hS^{1}}[1-d]
\] 
such that all vertical arrows in the arrow diagram \Cref{rightrelCYdiag} are isomorphisms. In particular, the map
$$f^{r} \Id{\Cbold} f \stackrel{f^{r} \overline{\tau} f}\arr \Serre{\Cbold}[1-d]$$
is required to be an isomorphism. In practice, the map $\Id{\Cbold} \stackrel{{\tau}}\arr \Serre{\Cbold}[1-d]$ is usually already an isomorphism, so that the boundary category $\Cbold$ is equipped with an {\it absolute right Calabi-Yau structure} of dimension $d-1$.

\sssec{} Here is the relation between relative co-orientations and relative right Calabi-Yau structures.

\begin{lem}\label{rightrelcytonccoor}

Suppose that we are given a relative right Calabi-Yau structure $\tau \in \Homm{\Abold}{\hoch{\Dbold \stackrel{f}\arr \Cbold}}{1_{\Abold}}^{hS^{1}}[1-d]$ of dimension $d$ on an $\Abold$-linear dualizable functor $\Dbold \stackrel{f}\arr \Cbold$ between proper categories, and that the restricted class $\overline{\tau} \in \Homm{\Abold}{\hoch{\Cbold}}{1_{\Abold}}^{hS^{1}}[1-d]$ is an absolute right Calabi-Yau structure on $\Cbold$. Then capping on $\tau \in \hoch{\Dbold \stackrel{f}\arr \Cbold}^{\vee}[1-d]$ gives an isomorphism $\hoch{\Dbold \stackrel{f}\arr \Cbold} \stackrel{- \cap \tau}\simeq \hoch{\Dbold \stackrel{f}\arr \Cbold}^{\vee}[1-d]$, hence a relative right Calabi-Yau structure induces a non-commutative relative co-orientation.

\end{lem}

\begin{proof}

Indeed, capping against $\overline{\tau}$ gives a map $\hochco{\Cbold} \stackrel{- \cap \overline{\tau}}\arr \hoch{\Cbold}^{\vee}[1-d]$, which is evidently an isomorphism since it is given by applying $\Homm{\Abold}{\Id{\Cbold}}{-}$ to the isomorphism $\Id{\Cbold} \stackrel{\overline{\tau}}\arr \Serre{\Cbold}[1-d]$. Then capping against $\tau$ gives a commutative diagram of fiber sequences
\begin{equation}
\xymatrix{\hochco{\Dbold; {\rm fib}(\eta)} \ar[r] \ar[d] & \hochco{\Dbold \stackrel{f}\arr \Cbold} \ar[r]\ar[d]^{- \cap \tau} & \hochco{\Cbold} \ar[d]^{\simeq} \\
\hoch{\Dbold}^{\vee}[-d] \ar[r] & \hoch{\Dbold \stackrel{f}\arr \Cbold}^{\vee}[1-d] \ar[r] & \hoch{\Cbold}^{\vee}[1-d].}
\end{equation}
We claim that the left vertical arrow is an isomorphism, hence the middle vertical arrow is as well, and hence we obtain a non-commutative relative orientation as claimed. Indeed, under the identification 
$$\hoch{\Dbold}^{\vee}[-d] \simeq \Hom{\Abold}{\Id{\Dbold}}{\Serre{\Dbold}}[-d],$$
the left hand vertical arrow is obtained by applying $\Homm{\Abold}{\Id{\Dbold}}{-}$ to the arrow ${\rm fib}(\eta) \arr \Serre{\Dbold}[-d]$, which is an isomorphism by the defining diagram \Cref{rightrelCYdiag} of a relative Calabi-Yau structure. 

\end{proof}

\ssec{Examples of non-commutative co-orientations}

We give a number of examples of relative right Calabi-Yau structures, hence by \Cref{rightrelcytonccoor} non-commutative relative co-orientations. As they are essentially dual to the above examples of non-commutative relative orientations, we shall be brief. The duality however is interesting, since it gives an alternative method of calculation in some examples.

\sssec{String topology}

As before, we fix a ground commutative ring spectrum $E$ and let ${\bf H}_{*}(X)=E \wedge X_{+}$ and ${\bf H}^{*}(X)=\Homm{\Spct}{\Sigma^{\infty}X_{+}}{E}$. Thus we are using (co)homology notation for the relevant spectra and suppressing the dependence on $E$. Given a finite type space $X$, the $E$-cohomology ${\bf H}^{*}(X)=\Endd{E}{E_{X}}$ is perfect as an $E$-module, so the category of modules $\LMod{{\bf H}^{*}(X)}$ is proper over $\Abold=\LMod{E}$. Moreover, since the constant local system $E_{X} \in \Loc{X}$ on a finite type space is compact, the category $\LMod{{\bf H}^{*}(X)}$ identifies with the full subcategory of $\Loc{X}$ generated under colimits by $E_{X}$.

As above, consider an oriented $d$-manifold with boundary $i : \partial M \arr M$ with relative orientation class $[M,\partial M] \in {\bf H}_{*}(M,\partial M)[-d]$ in the sense of Lefschetz duality. We claim that the pullback functor 
\[
i^{*} : \LMod{{\bf H}^{*}(M)} \arr \LMod{{\bf H}^{*}(\partial M)}
\]
carries an induced relative right Calabi-Yau structure hence by \Cref{rightrelcytonccoor} a non-commutative relative co-orientation of dimension $d$. We sketch the construction.

There are pairings $\Loc{M} \otimes \LMod{{\bf H}^{*}(M)} \arr \LMod{E}$ given by restricting the self-duality pairing
$$\Loc{M} \otimes \Loc{M} \stackrel{\Delta^{*}}\arr \Loc{M} \stackrel{p_{*}}\arr \LMod{E}$$
along the inclusion $\LMod{{\bf H}^{*}(M)} \subset \Loc{M}$, and likewise for the pairing
$$\Loc{\partial M} \otimes \LMod{{\bf H}^{*}(\partial M)} \arr \LMod{E}.$$
Since the pairings send a pair of compact objects to a perfect $E$-module, they are dualizable, hence we obtain induced pairings in Hochschild homology
$$\hoch{\Loc{M}} \otimes \hoch{{\bf H}^{*}(M)} \arr \LMod{E} \mbox{ and  } \hoch{\Loc{\partial M}} \otimes \hoch{{\bf H}^{*}(\partial M)} \arr \LMod{E},$$
hence circle-equivariant maps
$$\hoch{\Loc{M}} \arr \hoch{{\bf H}^{*}(M)}^{\vee} \mbox{ and } \hoch{\Loc{\partial M}} \arr \hoch{{\bf H}^{*}(\partial M)}^{\vee}.$$
By naturality of the pairings, we obtain moreover a circle-equivariant map
$$\hoch{\Loc{M},\Loc{\partial M}} \arr \Homm{E}{\hoch{{\bf H}^{*}(\partial M),{\bf H}^{*}(M)}}{E}.$$
Applied to a class $\theta \in \hoch{\Loc{M},\Loc{\partial M}}^{hS^{1}}[-d]$ we obtain a class
$$\tau \in \Homm{E}{\hoch{{\bf H}^{*}(\partial M),{\bf H}^{*}(M)}}{E}^{hS^{1}}[1-d].$$
If $\theta$ is a non-commutative relative orientation, then in fact $\tau$ will be a non-commutative relative co-orientation, by an argument formally dual to the case of a non-commutative relative orientation.

Assuming that $M$ and $\partial M$ are simply-connected, we may apply Koszul duality and Goodwillie-Jones in the simply-connected case, giving equivalences
$$\hochco{{\bf H}^{*}(M)} \simeq \hochco{\Loc{M}} \mbox{ and } \hoch{{\bf H}^{*}(M)} \simeq {\bf H}^{*}(LM),$$
and similarly
$$\hochco{{\bf H}^{*}(\partial M)} \simeq \hochco{\Loc{\partial M}} \mbox{ and } \hoch{{\bf H}^{*}(\partial M)} \simeq {\bf H}^{*}(L \partial M).$$
Thus applying \Cref{relcycdel}, we obtain the following.

\begin{cor}
Let $[M,\partial M] \in {\bf H}_{*}(M,\partial M)[-d]$ be a relative orientation of dimension $d$, with both $M$ and $\partial M$ simply connected. Then there is an induced framed $E_{2}$-algebra structure on the fiber sequence
$${\rm fib} \arr \hochco{{\bf H}^{*}(M) \arr {\bf H}^{*}(\partial M)} \arr \hochco{{\bf H}^{*}(\partial M)}.$$
Moreover, the underlying fiber sequence identifies under Goodwillie-Jones with the fiber sequence
$${\bf H}^{*}(LM)^{\vee}[-d] \arr {\bf H}^{*}(LM, L \partial M)^{\vee}[-d] \arr {\bf H}^{*}(L \partial M)^{\vee}[1-d].$$ 
\end{cor}


\sssec{Anti-canonical divisors}

Fix a proper Gorenstein integral scheme $X$ of dimension $d$ over a field $k$, a non-zero anti-canonical section $\sigma \in K_{X}^{-1}$, and the zero scheme $i: Z=Z(\sigma) \hookrightarrow X$. Note that both $\QCoh{X}$ and $\QCoh{Z}$ are proper. We claim that the pullback functor $i^{*} : \QCoh{X} \arr \QCoh{Z}$ has an induced relative non-commutative co-orientation of dimension $d$. We refer to
\cite[Appendix B.5]{preygel} for the relation between Hochschild invariants of ind-coherent sheaves and of quasi-coherent sheaves.

There are natural isomorphisms
$$\hoch{\IndCoh{X}} \simeq \hoch{\QCoh{X}}^{\vee} \mbox{ and } \hoch{\IndCoh{Z}} \simeq \hoch{\QCoh{Z}}^{\vee}$$
and in fact a commuting diagram of fiber sequences
\begin{equation}\label{anticankoszuldual}
\xymatrix{\hoch{\IndCoh{X},\IndCoh{Z}}[-1] \ar[r] \ar[d] & \hoch{\IndCoh{Z}} \ar[r] \ar[d] & \hoch{\IndCoh{X}}  \ar[d]\\
\hoch{\QCoh{Z},\QCoh{X}}^{\vee} \ar[r]  & \hoch{\QCoh{Z}}^{\vee} \ar[r]  & \hoch{\QCoh{X}}^{\vee},}
\end{equation}
where the arrow $\hoch{\IndCoh{Z}} \arr \hoch{\IndCoh{X}}$ is induced by pushforward of ind-coherent sheaves and the arrow $\hoch{\QCoh{Z}}^{\vee} \arr \hoch{\QCoh{X}}^{\vee}$ is dual to the arrow $\hoch{\QCoh{X}} \arr \hoch{\QCoh{Z}}$ induced by pullback. 

There are natural equivalences of Hochschild cohomologies
$$\hochco{\IndCoh{X}} \simeq \hochco{\QCoh{X}} \mbox{ and } \hochco{\IndCoh{Z}} \simeq \hochco{\QCoh{Z}}$$ as well as of the actions on the corresponding functors, hence there are equivalences of relative Hochschild cohomologies
$$\hochco{\IndCoh{Z} \arr \IndCoh{X}} \simeq \hochco{\QCoh{X} \arr \IndCoh{Z}}.$$
Letting
$$\theta \in \hoch{\IndCoh{X},\IndCoh{Z}}[-d]$$
be the relative non-commutative orientation of \Cref{ss:anticanon indcoh}, and applying the left hand vertical isomorphism in \Cref{anticankoszuldual}, we obtain a circle-invariant class
$$\tau \in \hoch{\QCoh{Z},\QCoh{X}}^{\vee}[1-d].$$
Given the natural identifications
$$\hochco{\IndCoh{X}} \simeq \hochco{\QCoh{X}} \mbox{ and } \hochco{\IndCoh{Z}} \simeq \hochco{\QCoh{Z}},$$ the fact that capping on $\tau$ is an isomorphism is equivalent to the fact that capping on $\theta$ is, and similarly for $\overline{\theta} \in \hoch{\IndCoh{Z}}[1-d]$ and $\overline{\tau} \in \hoch{\QCoh{Z}}^{\vee}[1-d]$. 

\medskip

Applying \Cref{relcycdel}, we obtain the following.

\begin{cor}
Let $X$ be a proper Gorenstein integral scheme  of dimension $d$ over a field $k$, with a non-zero anti-canonical section $\sigma \in K_{X}^{-1}$, and the zero scheme $i: Z=Z(\sigma) \hookrightarrow X$. Then the pullback functor $i^{*} : \QCoh{X} \arr \QCoh{Z}$ has a natural relative non-commutative co-orientation of dimension $d$, hence there is a fiber sequence of framed $E_{2}$-algebra structures
\[
{\rm fib} \arr \hochco{\QCoh{X} \arr \QCoh{Z}} \arr \hochco{\QCoh{Z}}
\]
with ${\rm fib}$ non-unital. The underlying fiber sequence of circle modules identifies with
\[
\hoch{\QCoh{Z},\QCoh{X}}^{\vee}[1-d] \arr \hoch{\QCoh{Z}}^{\vee}[1-d] \arr \hoch{\QCoh{X}}^{\vee}[1-d].
\]
\end{cor}

\bibliographystyle{amsalpha}
\bibliography{cyc-del-bib} 

\providecommand{\bysame}{\leavevmode\hbox to3em{\hrulefill}\thinspace}
\providecommand{\MR}{\relax\ifhmode\unskip\space\fi MR }
\providecommand{\MRhref}[2]{%
  \href{http://www.ams.org/mathscinet-getitem?mr=#1}{#2}
}
\providecommand{\href}[2]{#2}
\begin{thebibliography}{DCPR15}

\bibitem[AFT17]{AFT}
David Ayala, John Francis, and Hiro~Lee Tanaka, \emph{Factorization homology of
  stratified spaces}, Selecta Mathematica \textbf{23} (2017), no.~1, 293--362.

\bibitem[AMGR]{AMR}
David Ayala, Aaron Mazel-Gee, and Nick Rozenblyum, \emph{Factorization homology
  of enriched {$\infty$}-categories}, {\url{arXiv:1710.06414}}.

\bibitem[Bat08]{batanin}
Michael~A Batanin, \emph{The {E}ckmann--{H}ilton argument and higher operads},
  Advances in Mathematics \textbf{217} (2008), no.~1, 334--385.

\bibitem[BD19]{bravdyck}
Christopher Brav and Tobias Dyckerhoff, \emph{Relative {C}alabi--{Y}au
  structures}, Compositio Mathematica \textbf{155} (2019), no.~2, 372--412.

\bibitem[BD21]{bravdyck2}
\bysame, \emph{Relative {C}alabi--{Y}au structures {II}: shifted {L}agrangians
  in the moduli of objects}, Selecta Mathematica \textbf{27} (2021), no.~4, 63.

\bibitem[BK98]{barannkont}
Sergey Barannikov and Maxim Kontsevich, \emph{Frobenius manifolds and formality
  of {L}ie algebras of polyvector fields}, International Mathematics Research
  Notices \textbf{1998} (1998), no.~4, 201--215.

\bibitem[BR]{bravroz2}
C.~Brav and N.~Rozenblyum, \emph{Non-commutative calculus, connections, and
  loop spaces}, In preparation.

\bibitem[CLM07]{cohen}
Frederick~Ronald Cohen, Thomas~Joseph Lada, and Peter~J May, \emph{The homology
  of iterated loop spaces}, vol. 533, Springer, 2007.

\bibitem[Cos07]{costellotcft}
Kevin Costello, \emph{Topological conformal field theories and {C}alabi--{Y}au
  categories}, Advances in Mathematics \textbf{210} (2007), no.~1, 165--214.

\bibitem[CS99]{chassull}
Moira Chas and Dennis Sullivan, \emph{String topology}, arXiv preprint
  math/9911159 (1999).

\bibitem[DCPR15]{drummondstringtop}
Gabriel~C Drummond-Cole, Kate Poirier, and Nathaniel Rounds, \emph{Chain-level
  string topology operations}, arXiv preprint arXiv:1506.02596 (2015).

\bibitem[Dun88]{dunn}
Gerald Dunn, \emph{Tensor product of operads and iterated loop spaces}, Journal
  of Pure and Applied Algebra \textbf{50} (1988), no.~3, 237--258.

\bibitem[Ger63]{gerst}
Murray Gerstenhaber, \emph{The cohomology structure of an associative ring},
  Annals of Mathematics (1963), 267--288.

\bibitem[Get94]{getzlerbv}
Ezra Getzler, \emph{Batalin-{V}ilkovisky algebras and two-dimensional
  topological field theories}, Communications in mathematical physics
  \textbf{159} (1994), no.~2, 265--285.

\bibitem[Goo85]{goodwillie}
Thomas~G Goodwillie, \emph{Cyclic homology, derivations, and the free
  loopspace}, Topology \textbf{24} (1985), no.~2, 187--215.

\bibitem[GR17]{gr1}
Dennis Gaitsgory and Nick Rozenblyum, \emph{A {S}tudy in {D}erived {A}lgebraic
  {G}eometry: {V}olume i: {C}orrespondences and {D}uality}, Mathematical
  Surveys and Monographs, vol. 221, 2017.

\bibitem[GV95]{gerstvor}
Murray Gerstenhaber and A~Voronov, \emph{Higher operations on the {H}ochschild
  complex}, Funct. Anal. Appl \textbf{29} (1995), no.~1, 3.

\bibitem[Hor17]{horel}
Geoffroy Horel, \emph{Factorization homology and calculus {\`a} la {K}ontsevich
  {S}oibelman}, Journal of Noncommutative Geometry \textbf{11} (2017), no.~2,
  703--740.

\bibitem[HSS17]{hoyetal}
Marc Hoyois, Sarah Scherotzke, and Nicolo Sibilla, \emph{Higher traces,
  noncommutative motives, and the categorified {C}hern character}, Advances in
  Mathematics \textbf{309} (2017), 97--154.

\bibitem[Iri18]{irie}
Kei Irie, \emph{A chain level {B}atalin--{V}ilkovisky structure in string
  topology via de {R}ham chains}, International Mathematics Research Notices
  \textbf{2018} (2018), no.~15, 4602--4674.

\bibitem[Iwa20]{iwanari}
Isamu Iwanari, \emph{Differential calculus of {H}ochschild pairs for
  infinity-categories}, SIGMA. Symmetry, Integrability and Geometry: Methods
  and Applications \textbf{16} (2020), 097.

\bibitem[Jon87]{jones}
John~DS Jones, \emph{Cyclic homology and equivariant homology}, Inventiones
  mathematicae \textbf{87} (1987), no.~2, 403--423.

\bibitem[Kau08]{kauf}
Ralph~M Kaufmann, \emph{A proof of a cyclic version of {D}eligne's conjecture
  via cacti}, Mathematical Research Letters \textbf{15} (2008), no.~5,
  901--921.

\bibitem[KS06]{kontsoib}
Maxim Kontsevich and Yan Soibelman, \emph{Notes on {A}-infinity algebras,
  {A}-infinity categories and non-commutative geometry}, arXiv preprint
  math/0606241 (2006).

\bibitem[KTV21]{ktv1}
Maxim Kontsevich, Alex Takeda, and Yiannis Vlassopoulos,
  \emph{Pre-{C}alabi-{Y}au algebras and topological quantum field theories},
  arXiv preprint arXiv:2112.14667 (2021).

\bibitem[KTV23]{ktv2}
\bysame, \emph{Smooth {C}alabi-{Y}au structures and the noncommutative
  {L}egendre transform}, arXiv preprint arXiv:2301.01567 (2023).

\bibitem[KW21]{kellerwang}
Bernhard Keller and Yu~Wang, \emph{An introduction to relative {C}alabi-{Y}au
  structures}, arXiv preprint arXiv:2111.10771 (2021).

\bibitem[LD11]{luriefmp}
Jacob Lurie and X~DAG, \emph{Formal moduli problems}, Available at: http://www.
  math. ias. edu/~ lurie (2011).

\bibitem[Lur]{lurieHA}
Jacob Lurie, \emph{Higher algebra},
  {\url{http://www.math.harvard.edu/~lurie/papers/HA.pdf}}.

\bibitem[LZ93]{lianzuck}
Bong~H Lian and Gregg~J Zuckerman, \emph{New perspectives on the
  {BRST}-algebraic structure of string theory}, Communications in Mathematical
  Physics \textbf{154} (1993), no.~3, 613--646.

\bibitem[Man99]{maninfrob}
Yu.~I. Manin, \emph{Frobenius manifolds, quantum cohomology, and moduli
  spaces}, vol.~47, American Mathematical Soc., 1999.

\bibitem[May06]{may}
J~Peter May, \emph{The geometry of iterated loop spaces}, vol. 271, Springer,
  2006.

\bibitem[MS02]{MS-Deligne}
James~E. McClure and Jeffrey~H. Smith, \emph{A solution of {D}eligne's
  {H}ochschild cohomology conjecture}, Recent progress in homotopy theory
  ({B}altimore, {MD}, 2000), Contemp. Math., vol. 293, Amer. Math. Soc.,
  Providence, RI, 2002, pp.~153--193.

\bibitem[Pre11]{preygel}
Anatoly Preygel, \emph{Thom-{S}ebastiani \& duality for matrix factorizations},
  arXiv preprint arXiv:1101.5834 (2011).

\bibitem[PS92]{penkschw}
Michael Penkava and Albert Schwarz, \emph{On some algebraic structures arising
  in string theory}, arXiv preprint hep-th/9212072 (1992).

\bibitem[SW03]{SW}
Paolo Salvatore and Nathalie Wahl, \emph{Framed discs operads and
  {B}atalin--{V}ilkovisky algebras}, Quarterly Journal of Mathematics
  \textbf{54} (2003), no.~2, 213--231.

\bibitem[Tam98]{tamarkinformality}
Dmitry~E Tamarkin, \emph{Another proof of {M.} {K}ontsevich formality theorem},
  arXiv preprint math/9803025 (1998).

\bibitem[Tam07]{tamarkindg}
Dmitry Tamarkin, \emph{What do dg-categories form?}, Compositio Mathematica
  \textbf{143} (2007), no.~5, 1335--1358.

\bibitem[TT00]{tamtsyg}
Dmitry Tamarkin and Boris~L Tsygan, \emph{Differential calculus, homotopy {BV}
  algebras, and formality conjectures}, Methods of Functional Analysis and
  Topology \textbf{6} (2000).

\bibitem[TZ06]{tradzein}
Thomas Tradler and Mahmoud Zeinalian, \emph{On the cyclic {D}eligne
  conjecture}, Journal of Pure and Applied Algebra \textbf{204} (2006), no.~2,
  280--299.

\bibitem[War12]{ward}
Benjamin~C Ward, \emph{Cyclic {$A_\infty$} structures and {D}eligne's
  conjecture}, Algebraic \& Geometric Topology \textbf{12} (2012), no.~3,
  1487--1551.

\bibitem[Wit90]{wittenantibracket}
Edward Witten, \emph{A note on the antibracket formalism}, Modern Physics
  Letters A \textbf{5} (1990), no.~07, 487--494.

\bibitem[Yeu16]{yeung}
Wai-kit Yeung, \emph{Relative {C}alabi-{Y}au completions}, arXiv preprint
  arXiv:1612.06352 (2016).

\end{thebibliography}

\end{document}